\documentclass{amsart}
\usepackage{tikz-cd}
\usepackage{latexsym,amssymb}
\usepackage[english]{babel}
\usepackage{bm}
\usepackage{amsfonts, amsthm}
\usepackage{amsmath}
\usepackage{mathrsfs}
\usepackage{multirow}
\usepackage{bm}
\usepackage{pgf,tikz}
\usepackage{hyperref}

\newtheorem{prop}{Proposition}[section]
\newtheorem{lem}[prop]{Lemma}

\newtheorem{thm}[prop]{Theorem}
\newtheorem{conj}[prop]{Conjecture}
\theoremstyle{definition}
\newtheorem{rem}[prop]{Remark}
\newtheorem{defi}[prop]{Definition}
\newtheorem{ex}[prop]{Example}

\def\Z{\mathbb{Z}}

\numberwithin{equation}{section}

\def\Z{\mathbb{Z}}
\def\N{\mathbb{N}}

\def\G{\Gamma}

\def\H{\mathrm{H}}
\def\N{\mathrm{N}}

\def\Cay{\mathsf{Cay}}

\def\E{\mathcal{E}}
\def\D{\mathcal{D}}

\def\C{\mathcal{C}}

\begin{document}

\title{Tight globally simple non-zero sum Heffter arrays and biembeddings}

\author[L. Mella]{Lorenzo Mella}
\address{Dip. di Scienze Fisiche, Informatiche, Matematiche, Universit\`a degli Studi di Modena e Reggio Emilia, Via Campi 213/A, I-41125 Modena, Italy}
\email{lorenzo.mella@unipr.it}

\author[A. Pasotti]{Anita Pasotti}
\address{DICATAM - Sez. Matematica, Universit\`a degli Studi di Brescia, Via
Branze 43, I-25123 Brescia, Italy}
\email{anita.pasotti@unibs.it}

\begin{abstract}
Square relative non-zero sum Heffter arrays, denoted by $\N\H_t(n;k)$, have been introduced as a variant of the classical concept of Heffter array.
An $\N\H_t(n; k)$ is an $n\times n$ partially filled  array  with elements in $\Z_v$, where
$v=2nk+t$, whose rows and whose columns contain $k$ filled cells,
such that the sum of the elements in every row and column is different from $0$
(modulo $v$) and, for every $x\in \Z_v$ not belonging to the subgroup of
order $t$, either $x$ or $-x$ appears in the array.
In this paper we give direct constructions of square non-zero sum Heffter arrays with no empty cells, $\N\H_t(n;n)$, for every $n$ odd, when 
$t$ is a divisor of $n$ and when $t\in\{2,2n,n^2,2n^2\}$.
The constructed arrays have  also the very restrictive property of being ``globally simple'';
this allows us to get  new orthogonal path decompositions and new biembeddings of complete multipartite graphs.
\end{abstract}


\maketitle

\section{Introduction}
An $m\times n$ partially filled (p.f., for short) array on a set $\Omega$ is an $m\times n$ matrix whose
elements belong to $\Omega$ and where some cells can be empty.
In \cite{A}, Archdeacon introduced a class of p.f. arrays on a cyclic group, called \emph{Heffter arrays}, which have been extensively studied since they have a vast variety of applications and connections with other well known problems and concepts. Such arrays have been generalized as follows, see \cite{ CPEJC}.

\begin{defi}\label{def:lambdaRelative}
Let $v=\frac{2nk}{\lambda}+t$ be a positive integer,
where $t$ divides $\frac{2nk}{\lambda}$,  and
let $J$ be the subgroup of $\Z_{v}$ of order $t$.
 A  $\lambda$-\emph{fold Heffter array $A$ over $\Z_{v}$ relative to $J$}, denoted by $^\lambda\H_t(m,n; h,k)$, is an $m\times n$ p.f.  array
 with elements in $\Z_{v}$ such that:
\begin{itemize}
\item[($\rm{a})$] each row contains $h$ filled cells and each column contains $k$ filled cells;
\item[($\rm{b})$] the multiset $\{\pm x \mid x \in A\}$ contains  each element of $\Z_v\setminus J$ exactly $\lambda$ times;
\item[($\rm{c})$] the elements in every row and column sum to $0$ (in $\Z_v$).
\end{itemize}
\end{defi}

In the following, by $t$ admissible we mean that $t$ divides $\frac{2nk}{\lambda}$.
If $t=\lambda =1$ one finds again the concept introduced in \cite{A}.
Also if $t=1$ or $\lambda=1$, then it is omitted, for example a $^1 \H_1(m,n;h,k)$ is simply denoted by $\H(m,n;h,k)$.
Note that $nk=mh$, so if the array is square then $h=k$ and by $^\lambda \H_t(n;k)$ one means a $^\lambda \H_t(n,n;k,k)$.
The most important result on Heffter arrays with $\lambda=t=1$ is the following one, see \cite{ADDY,CDDY,DW}.
\begin{thm}
For every $n\geq k\geq3$ there exists an $\H(n;k)$.
\end{thm}
We have to emphasize that Heffter arrays are not only combinatorial objects interesting \emph{per se}, in fact
there are several papers in which they are investigated to get face $2$-colorable embeddings, namely biembeddings, see \cite{A, CDY,  CMPPHeffter,  CPPBiembeddings, CPJCTA, CPEJC, DM}.

In a very recent paper, see \cite{CDFP}, Costa, Della Fiore and Pasotti have introduced
a variant of Heffter arrays, that they called \emph{non-zero sum Heffter arrays}
since they require that all the properties of Heffter arrays are satisfied except that on row and column sums which now have to be all different from zero.

\begin{defi}\label{def:NZS}
Let $v=\frac{2nk}{\lambda}+t$ be a positive integer,
where $t$ divides $\frac{2nk}{\lambda}$,  and
let $J$ be the subgroup of $\Z_{v}$ of order $t$.
 A   $\lambda$-\emph{fold non-zero sum Heffter array $A$ over $\Z_{v}$ relative to $J$}, denoted by $^\lambda \N\H_t(m,n; h,k)$, is an $m\times n$ p.f.  array
 with elements in $\Z_{v}$ such that:
\begin{itemize}
\item[($\rm{a_1})$] each row contains $h$ filled cells and each column contains $k$ filled cells;
\item[($\rm{b_1})$] the multiset $\{\pm x \mid x \in A\}$ contains each element of $\Z_v\setminus J$ exactly $\lambda$ times;
\item[($\rm{c_1})$] the sum of the elements in every row and column is different from $0$ (in $\Z_v$).
\end{itemize}
\end{defi}

Also this variant is related to several other topics, in fact in
\cite{CDFP} the authors explain its connection to a famous conjecture by Alspach, to orthogonal path decompositions
and to biembeddings. Furthermore, they attack the existence problem
and, using a probabilistic approach, they provide a complete (non constructive) solution for an $\N\H_t(m,n;h,k)$ for any admissible value of the parameters.
Moreover, they construct square  non-zero sum Heffter arrays with empty cells and rectangular ones with no empty cells that satisfy the very restrictive property of being ``globally simple''  when $\lambda=t=1$.
Other results on  relative non-zero sum Heffter arrays  can be found in \cite{C, M}.

Also in this paper we focus on non-zero sum Heffter arrays having this additional property of being ``globally simple'',
here we take $\lambda=1$, $t>1$ and we suppose that the array is \emph{tight}, that is with no empty cells.
In fact, after having introduced some basic definitions in Section 2, in Section 3 we consider totally filled square non-zero sum Heffter arrays.
In particular we present several classes of
globally simple $\N\H_t(n;n)$ for $n$ odd which lead to a complete solution to the existence problem for every $n$ prime
and admissible $t$, and for every odd $n$ and   $t$ divisor of $n$.
Finally in Sections 4 and 5
 we explain how relative non-zero sum Heffter arrays can be used to construct cyclic path decompositions
of the complete multipartite graph and
biembeddings of such decompositions into an orientable surface, respectively. As a consequence, thanks to the existence results described in Section 3, we provide new classes of such decompositions and biembeddings.

\section{Simple orderings}
Let $a,b$ be two integers and suppose $a\leq b$, by $[a,b]$ we denote the set $\{a,a+1,a+2,\ldots,b\}$.
Also, given a set $S = \{a_1,a_2, \dotsc, a_k\}$ and an integer $x$, by $S+x$ we mean $\{a_1+x,a_2+x, \dotsc, a_k+x\}$.

Given a finite subset $T$ of an abelian group $G$ with $\sum_{t \in T}t\neq 0$
and an ordering $\omega=(t_1,t_2,\ldots,t_k)$
of the elements of $T$, let $s_i=\sum_{j=1}^i t_j$, for any $i \in [1,k]$, be the $i$-th partial sum
of $\omega$ and let $S(\omega)=\{s_1,\ldots,s_k\}$ be the unordered list of the partial sums
of $T$ computed with respect to $\omega$. The ordering $\omega$ is \emph{simple} if
all the partial sums are non-zero and pairwise distinct or, in other words, if there is no proper subsequence of $\omega$
that sums to $0$. Note that if $\omega$ is a simple ordering so is $\omega^{-1}=(t_k,t_{k-1},\ldots,t_1)$.

If $A$ is an $m \times n$ p.f. array, its rows and columns are denoted by $R_1,\ldots, R_m$ and by $C_1,\ldots, C_n$,
respectively. By $\E(A)$ we mean the unordered list of the elements of the filled cells of $A$.
Analogously, by
$\E(R_i)$ and $\E(C_j)$ we denote the unordered lists of elements of the $i$-th row and of the $j$-th column,
respectively, of $A$.
By $\omega_{R_i}$ and $\omega_{C_j}$ we mean a simple ordering of the $i$-th row and of the $j$-th column, respectively.
If $A$ is a non-zero sum Heffter array, by a \emph{natural ordering}
of a row (column) of $A$ we mean the ordering from left to right (from top to bottom)
starting from a given element of such a row (column).
Also,  $A$  is said
\begin{itemize}
\item \emph{simple} if each row and each column of $A$ admits a simple ordering;
\item \emph{globally simple} if  each row and each column of $A$ admits a simple natural ordering.
\end{itemize}

\begin{rem}\label{rem:1}
Note that there is an important difference in the definition of the property of being globally simple
for Heffter arrays (see \cite{BCDY, CMPPHeffter}) and
for non-zero sum  Heffter arrays. In fact for the first ones each row and each column has exactly one natural ordering, while
if we have a non-zero sum Heffter array the number of natural orderings of a row (column) is equal to the number of its filled cells, since it depends on the starting element. Suppose for instance that $R=(1,10,4,-11)$ is a row of an $\N\H(4;4)$.
If we compute its partial sums with respect to the natural ordering starting from $1$ we get $S(R)=(1,11,15,4)$, and hence the ordering is simple,
while if we compute its partial sums with respect to the natural ordering starting from $4$ we get $S(R)=(4,-7,-6,4)$
and hence the ordering is not simple.
So it is clear that when we compute the partial sums of a row or of a column we have always to indicate the starting element.
In the following, given an array $A=(a_{i,j})$, with a little abuse of notation, by $S(R_i)$ ($S(C_j)$) we will denote the unordered list of the partial sums
of the row $R_i$ (the column $C_j$) with respect to the natural ordering calculated starting from the element $a_{i,1}$ ($a_{1,j}$).
\end{rem}

\begin{rem}\label{rem:2}
If a row or a column admits a simple natural ordering, then all its partial sums are different from zero and hence,
in particular, also the sum of all its elements is different from $0$.
Hence in our constructions of globally simple non-zero sum Heffter arrays first we prove only that conditions ($\rm{a_1})$ and ($\rm{b_1})$ of Definition \ref{def:NZS} are satisfied. Then
we show that each row and each column admits a simple natural ordering, obtaining as a consequence that also condition
($\rm{c_1})$ is satisfied.
\end{rem}

Now we can explain the connection between non-zero sum Heffter arrays and the following conjecture proposed by Alspach.
\begin{conj}\label{conj}
Let $T\in \Z_v\setminus \{0\}$ with $\sum_{t \in T}t\neq 0$, then there exists a simple ordering of $T$.
\end{conj}
Clearly, if it holds, then every non-zero sum Heffter array is simple.
Several partial results on Conjecture \ref{conj} have been obtained, see \cite{CDFOR} and the references therein.

\section{Constructions of globally simple $\N\H_t(n;n)$, $n$ odd}\label{sec:constructions}
In this section we consider tight square relative non-zero sum Heffter arrays, namely $\N\H_t(n;n)$;
in details for every odd $n\geq1$ we present direct constructions of a globally simple $\N\H_t(n;n)$
for $t \in\{2,2n,n^2,2n^2\}$ and whenever $t$ divides $n$.
In particular, this will allow us to give a complete solution
to the existence problem of a globally simple $\N\H_t(n;n)$ for every prime $n$ and every admissible $t$.

To present these constructions we have to introduce the concept of \emph{support} of an array, of a row and of a column.
Given a relative non-zero sum Heffter array $\N\H_t(n;n)$, say $A$, and viewed its elements as integers in
$\pm \left\{1,\ldots, \left\lfloor \frac{2n^2+t}{2}\right\rfloor\right\}$ the \emph{support} of $A$, denoted by \emph{supp}($A$),
is defined to be the set of absolute values of the elements contained in $A$. Analogously one can define
the support of a row and of a column of $A$. Note that if $A$ is an $\N\H_t(n;n)$, then
its elements belong to $\Z_{2n^2+t}$ and we have to avoid the elements of the subgroup of $\Z_{2n^2+t}$
of order $t$, denoted by $\frac{2n^2+t}{t}\Z_{2n^2+t}$.

The following lemma will play a fundamental role in showing that the arrays we are going to construct are globally simple.
Its proof is trivial and hence it is left to the reader.

\begin{lem}\label{lemma:main}
Let $a,\, b$ and $g$ be elements in an abelian group $G$, and let $\ell$ be a positive integer. Consider the orderings $\omega = (a,b,a+g,b-g, a+2g, b-2g, \dotsc,b - (\ell-1)g, a+\ell g)$ and $\nu = (a,b,a+g,b-g, a+2g, b-2g, \dotsc,a+\ell g, b-\ell g)$. Then:
\begin{enumerate}
    \item the unordered list of the partial sums of the ordering $ \omega $ is
\[
S(\omega)=\{ k(a+b): 1\leq k \leq \ell \} \cup \{a + k(a+b+g): 0\leq k \leq \ell \};
\]
\item the unordered list of the partial sums of the ordering  $\omega^{-1}$ is
\[
S(\omega^{-1})=\{ k(a+b+g): 1\leq k \leq \ell \} \cup \{a+\ell g + k(a+b): 0\leq k \leq \ell\};
\]
\item the unordered list of the partial sums of the ordering $ \nu $ is
\[
S(\nu)=\{ k(a+b): 1\leq k \leq \ell+1 \} \cup \{a + k(a+b+g): 0\leq k \leq \ell \};
\]
\item the unordered list of the partial sums of the ordering  $\nu^{-1}$ is
\[
S(\nu^{-1})=\{ k(a+b): 1\leq k \leq \ell+1\} \cup \{b-\ell g + k(a+b+g): 0\leq k \leq \ell \}.
\]
\end{enumerate}
\end{lem}

We point out that all the following proofs are constructive, in fact we give direct constructions
of a globally simple $\N\H_{t}(n;n)$, for suitable values of $n$ and $t$, and we show that all the required properties are satisfied, keeping in mind Remarks
\ref{rem:1} and \ref{rem:2}.

\subsection{Globally simple $\N\H_{2}(n;n)$}
The arrays described in the following proposition can be found in \cite[Theorem 5.5]{CDFP}, where a construction of a globally simple $\N\H (m,n;n,m)$ for every $m,n \geq 1$ is given. Here, we show that for every integer $n = m$ the same array is also a globally simple $\N\H_2(n;n)$.
We underline that in this construction we do not require $n$ odd.
\begin{prop}\label{prop:2}
For every integer $n\geq 1$, there exists a globally simple $\N\H_2(n;n)$.
\end{prop}
\begin{proof}
Let $n \geq 1$ be an integer, and consider the $n \times n$ matrix $A=(a_{i,j})$ with elements in $\Z_{2n^2+2}$ whose $(i,j)$-th entry is:
\[
a_{i,j} = \varepsilon [j+(i-1)n],
\]
where $\varepsilon=1$ if $i \equiv j \pmod{2}$, and $\varepsilon = -1$ otherwise.
Since the array is tight and, as shown in \cite{CDFP}, $supp(A) = [1,n^2]$, conditions $(\rm{a_1})$ and $(\rm{b_1})$  of Definition \ref{def:NZS} are satisfied.

It now remains to check that the array satisfies condition $(\rm{c_1})$ of Definition \ref{def:NZS}
and it is globally simple.
As proved in \cite{CDFP},  the unordered list of the partial sums of the $j$-th column of $A$, for $j \in [1,n]$, is $S(C_j) = X_j \cup Y_j$, where:
\[
\begin{aligned}
X_j    &= \left\{ (-1)^j kn : k \in \left[1, \left\lfloor \frac{n}{2}\right\rfloor\right]   \right\}, \\
Y_j &= \left\{  (-1)^{j+1} (j+kn): k \in \left[0, \left\lfloor \frac{n-1}{2}\right\rfloor \right] \right\}.
\end{aligned}
\]
It can be easily seen that  $X_j$ and $Y_j$ are disjoint sets modulo $2n^2+2$, containing only non-zero elements.

Now, for every $i \in [1,n]$ the unordered list of partial sums of the $i$-th row of $A$ is $S(R_i) = X_i \cup Y_i$, where:
\[
\begin{aligned}
X_i    &= \left\{ (-1)^i k  : k \in \left[1, \left\lfloor \frac{n}{2}\right\rfloor\right]   \right\}, \\
Y_i &= \left\{  (-1)^{i+1} (1+(i-1)n+k): k \in \left[0, \left\lfloor \frac{n-1}{2}\right\rfloor \right] \right\}.
\end{aligned}
\]
As before, it can be  seen that $X_i$ and $Y_i$ are disjoint sets modulo $2n^2+2$, and that no element in these sets is zero.

Hence,  for every integer $n \geq 1$ the array is globally simple, so, in particular,  every row and column  has non-zero sum, proving
that also condition $(\rm{c_1})$ of Definition \ref{def:NZS} is satisfied.
\end{proof}

\begin{rem}\label{rem:sum_2}
From the globally simple array $A$ shown in Proposition \ref{prop:2}, the total sum of the elements of each row and of each column can be easily found, minding the parity of $n$.
For every odd integer $n \geq 1$, considering the element of $Y_j$ and that of $Y_i$ for $k=\frac{n-1}{2}$, we get that for every $i,\,j \in [1,n]$:
\[
\begin{aligned}
\sum_{a_{i,j} \in C_j} a_{i,j}  &=(-1)^{j+1} \left( j+  \frac{n-1}{2}\,n  \right); \\
\sum_{a_{i,j} \in R_i} a_{i,j} &=(-1)^{i+1} \left( (i-1)n  + \frac{n+1}{2} \right). \\
\end{aligned}
\]
Analogously, for every even integer $ n \geq 1$, considering the element of $X_j$ and that of $X_i$ with $k=\frac{n}{2}$, we obtain that for every $i,\,j \in [1,n]$:
\[
\begin{aligned}
\sum_{a_{i,j} \in C_j} a_{i,j}  &= (-1)^j \frac{n^2}{2} ; \\
\sum_{a_{i,j} \in R_i} a_{i,j} &= (-1)^i \frac{n}{2} .  \\
\end{aligned}
\]
\end{rem}

\begin{ex}

Following the proof of Proposition \ref{prop:2}, we obtain the $\N\H_{2}(11;11)$ below, whose elements belong to $\Z_{244}$:

\begin{center}
\begin{footnotesize}
$\begin{array}{|r|r|r|r|r|r|r|r|r|r|r|}\hline
1 & -2 & 3 & -4 & 5 & -6 & 7 & -8 & 9 & -10 & 11\\ \hline
-12 & 13 & -14 & 15 & -16 & 17 & -18 & 19 & -20 & 21 & -22\\ \hline
23 & -24 & 25 & -26 & 27 & -28 & 29 & -30 & 31 & -32 & 33\\ \hline
-34 & 35 & -36 & 37 & -38 & 39 & -40 & 41 & -42 & 43 & -44\\ \hline
45 & -46 & 47 & -48 & 49 & -50 & 51 & -52 & 53 & -54 & 55\\ \hline
-56 & 57 & -58 & 59 & -60 & 61 & -62 & 63 & -64 & 65 & -66\\ \hline
67 & -68 & 69 & -70 & 71 & -72 & 73 & -74 & 75 & -76 & 77\\ \hline
-78 & 79 & -80 & 81 & -82 & 83 & -84 & 85 & -86 & 87 & -88\\ \hline
89 & -90 & 91 & -92 & 93 & -94 & 95 & -96 & 97 & -98 & 99\\ \hline
-100 & 101 & -102 & 103 & -104 & 105 & -106 & 107 & -108 & 109 & -110\\ \hline
111 & -112 & 113 & -114 & 115 & -116 & 117 & -118 & 119 & -120 & 121\\ \hline
\end{array}$
\end{footnotesize}
\end{center}

And similarly, we obtain the $\N\H_2(10;10)$ below, filled with elements in $\Z_{202}$:
\begin{center}
\begin{footnotesize}
$\begin{array}{|r|r|r|r|r|r|r|r|r|r|}\hline
1 & -2 & 3 & -4 & 5 & -6 & 7 & -8 & 9 & -10\\ \hline
-11 & 12 & -13 & 14 & -15 & 16 & -17 & 18 & -19 & 20\\ \hline
21 & -22 & 23 & -24 & 25 & -26 & 27 & -28 & 29 & -30\\ \hline
-31 & 32 & -33 & 34 & -35 & 36 & -37 & 38 & -39 & 40\\ \hline
41 & -42 & 43 & -44 & 45 & -46 & 47 & -48 & 49 & -50\\ \hline
-51 & 52 & -53 & 54 & -55 & 56 & -57 & 58 & -59 & 60\\ \hline
61 & -62 & 63 & -64 & 65 & -66 & 67 & -68 & 69 & -70\\ \hline
-71 & 72 & -73 & 74 & -75 & 76 & -77 & 78 & -79 & 80\\ \hline
81 & -82 & 83 & -84 & 85 & -86 & 87 & -88 & 89 & -90\\ \hline
-91 & 92 & -93 & 94 & -95 & 96 & -97 & 98 & -99 & 100\\ \hline
\end{array}$
\end{footnotesize}
\end{center}
\end{ex}

\subsection{Globally simple $\N\H_{2n}(n;n)$}

\begin{prop}\label{prop:2n}
For every odd integer $n \geq 1$ there exists  a globally simple $\N\H_{2n}(n;n)$.
\end{prop}
\begin{proof}
Let $n\geq 1$ be odd and consider the $n \times n$ array $A$ with elements in $\Z_{2n^2+2n}$ whose $(i,j)$-th entry is:
\[
a_{i,j}= \left\{\begin{aligned}
&i + (n+1)(j-1) \qquad &\text{for $i\equiv j \pmod{2}$,} \\
&n^2+n-i - (n+1)(j-1) \qquad &\text{for $i\not\equiv j \pmod{2}$.} \\
\end{aligned}\right.
\]
$A$ has no empty cells, thus condition ($\rm{a_1})$ of Definition \ref{def:NZS} is trivially satisfied.
It can then be verified that, for $j \in \left[1,\frac{n-1}{2}\right]$, the support of the columns is:
\[
\begin{aligned}
supp(C_j \cup C_{n-j+1}) &= \left[1+(n+1)(j-1),n + (n+1)(j-1)\right] \cup \\ &\left[n^2 - (n+1)(j-1), n^2+n-1 -(n+1)(j-1) \right]
\end{aligned}
\]
and that
$$supp\left(C_{\frac{n+1}{2}}\right) = \left[1+\frac{n^2-1}{2}, n+\frac{n^2-1}{2}\right],$$
from which we obtain that $supp(A) =[1,n^2+n]\setminus \{n+1,2(n+1),\ldots, n(n+1)\}$.

We now examine the partial sums of the columns of $A$.
For any odd $j\in [1,n]$ the $j$-th column of $A$ is given by:
\[
\begin{aligned}
C_j = &\big( 1+(n+1)(j-1), n^2+n-2-(n+1)(j-1), 3+ (n+1)(j-1), \\
& n^2+n-4-(n+1)(j-1),\dotsc, n^2+n-(n-1)-(n+1)(j-1),\\
&  n+ (n+1)(j-1)\big).
\end{aligned}
\]
Define now:
\[
\begin{aligned}X &= \left\{k(n^2+n-1): \, k \in\left[1,\frac{n-1}{2}\right] \right\}, \\  Y_j &= \left\{1+(n+1)(j-1)+k(n^2+n+1):\,
k \in \left[0,\frac{n-1}{2}\right] \right\}.
\end{aligned}
\]
  By applying Lemma \ref{lemma:main}(1) with $a = 1+(n+1)(j-1)$, $b = n^2+n-2-(n+1)(j-1)$, $g=2$ and $\ell=\frac{n-1}{2}$, it can be seen that for any column $C_j$, with $j$ odd, it results
$S(C_j) = X\cup Y_j$.
We notice that in $\Z_{2n^2+2n}$ both $X$ and $Y_j$ are sets not containing zero. It remains to check that $X$ and $Y_j$ are disjoint. Assume by contradiction that there exist $k_1 \in\left[1,\frac{n-1}{2}\right]$, $k_2 \in \left[0,\frac{n-1}{2}\right] $ and an odd $j$ such that:
\[
k_1(n^2+n-1) \equiv 1+(n+1)(j-1)+k_2(n^2+n+1) \pmod{2n^2+2n}.
\]
It is easy to see that this would imply:
\[
k_1+k_2+1 \equiv 0 \pmod{n+1}.
\]
However, by the hypothesis on $k_1$ and $k_2$ we have $1 \leq k_1+k_2 \leq n-1$, so we obtain a contradiction. Then, the partial sums of $C_j$, $j$ odd, are pairwise distinct, and non-zero.

 It can be checked that for any even $j\in[1,n]$ the $j$-th column of $A$ is:
\[
\begin{aligned}
C_j = &\big( n^2+n-1-(n+1)(j-1), 2+(n+1)(j-1), n^2+n-3-(n+1)(j-1), \\ &4+(n+1)(j-1),\dotsc,
n-1+(n+1)(j-1), n^2-(n+1)(j-1)\big).
\end{aligned}
\]

Now, consider:
\[
\begin{aligned}
X &= \left\{k(n^2+n+1): \, k\in \left[1,\frac{n-1}{2}\right]\right\}, \\  Y_j &= \left\{n^2-2-(n+1)(j-2)+k(n^2+n-1):\,  k \in \left[0, \frac{n-1}{2}\right] \right\}.
\end{aligned}
\]
By applying Lemma \ref{lemma:main}(1) with $a = n^2+n-1-(n+1)(j-1)$, $b = (n+1)(j-1)+2$, $g = -2$ and $\ell=\frac{n-1}{2}$, it can be seen that for any column $C_j$, with $j$ even, we have
$S(C_j) = X\cup Y_j.$
Reasoning as before, one can see that $X$ and $Y_j$ are disjoint sets not containing zero, hence the partial sums are pairwise distinct and non-zero.

Now we consider the partial sums of the rows. It is easy to see
 that for any odd $i \in [1,n]$ the $i$-th row of the matrix is:
\[
\begin{aligned}
R_i = &\big( i, n^2-1-i, i+ 2(n+1), n^2-1-i-2(n+1), i + 4(n+1), \dotsc, \\
 &n^2-1-i-(n-3)(n+1), i + (n-1)(n+1) \big). \\
\end{aligned}
\]
Then, set:
    \[
    \begin{aligned}
    X &= \left\{k(n^2-1): \, k\in \left[1,\frac{n-1}{2}\right] \right\},\\
    Y_i &= \left\{i+k(n^2+2n+1):\,  k\in\left[0,\frac{n-1}{2}\right] \right\}. \\
    \end{aligned}
    \]
    By applying Lemma \ref{lemma:main}(1) with $a= i$, $b = n^2-1-i$, $g = 2n+2$ and $\ell=\frac{n-1}{2}$, it can be seen that for any row $R_i$, $i$ odd, it results
  $ S(R_i) = X\cup Y_i.$
		  Since we are working in $\Z_{2n^2+2n}$, we can rewrite $Y_i$ as:
    $$Y_i =\left\{i+k(1-n^2): \, k\in \left[0,\frac{n-1}{2}\right] \right\}.$$
	Now, we have to show that $S(R_i)$ is a set and that $0 \not\in S(R_i)$.  We first notice that in $\Z_{2n^2+2n}$ both $X$ and $Y_i$ are sets not containing zero.
By way of contradiction, assume that for some $k_1 \in \left[1,\frac{n-1}{2}\right]$, $k_2 \in \left[0,\frac{n-1}{2}\right]$ and odd $i$ we have:
    \[
    k_1(n^2-1) \equiv i+k_2(1-n^2) \pmod{2n^2+2n}.
    \]
    This would imply that:
    \[
    (k_1+k_2)(n^2-1) \equiv (k_1+k_2)(n-1)(n+1) \equiv i \pmod{2n^2+2n}.
    \]
    It would then follow that $i$ is a multiple of $n+1$; however, $i\in [1,n]$. We have then proven that all the partial sums are distinct and  non-zero.

Finally, it is not hard to see that for any even $i \in [1,n]$  the $i$-th row of $A$ is:
\[
\begin{aligned}
R_i = &\big(n^2+n-i, i+ (n+1), n^2+n-i-2(n+1), i+ 3(n+1),\dotsc, \\
&i+(n-2)(n+1), n^2+n-i-(n+1)(n-1)  \big). \\
\end{aligned}
\]
Now, let:
    \[
    \begin{aligned}
    X &=\left\{k(n^2+2n+1): \, k \in\left[ 1,\frac{n-1}{2}\right] \right\},\\  Y_i &= \left\{n^2+n-i+k(n^2-1):\,  \left[ 0,\frac{n-1}{2}\right] \right\}.
    \end{aligned}
    \]
    By applying Lemma \ref{lemma:main}(1) with $a = n^2+n-i$, $b = n+1+i$, $g=-2n-2$ and $\ell=\frac{n-1}{2}$, it can be seen that for any row $R_i$, $i$ even, we have $ S(R_i) = X\cup Y_i.$
 Since, as before, we are working in $\Z_{2n^2+2n}$, we can rewrite $X$ as:
    $$X =\left\{k(1-n^2): \, k\in \left[1,\frac{n-1}{2}\right] \right\}.$$
    With an argument similar to the previous one we can show that $X$ and $Y_i$ are disjoint sets not containing zero, proving that the rows of
		$A$ admit a  simple  natural ordering. This concludes the proof.
\end{proof}

\begin{rem}\label{rem:sum_2n}
Let $A$ be the globally simple array described in Proposition \ref{prop:2n}.
For every odd $j \in [1,n]$ looking at the element of the corresponding $Y_j$ obtained for $k=\frac{n-1}{2}$ we get
\[
\sum_{a_{i,j} \in C_j} a_{i,j}  =j(n+1)+\frac{n^3-2n-1}{2} = (n+1)\left(\frac{n^2-n-1}{2}+j\right);
\]
while for every even $j \in [1,n]$:
\[
\sum_{a_{i,j} \in C_j} a_{i,j}  =-j(n+1)+\frac{n^3+2n^2+2n+1}{2} = (n+1)\left(\frac{n^2+n+1}{2}-j\right).
\]
It can be seen that the previous expressions can be written in a more compact form, obtaining that for every $j \in [1,n]$:
 \[
 \sum_{a_{i,j} \in C_j} a_{i,j}  = (n+1)\left(\frac{n^2}{2} +(-1)^j\,\frac{n+1-2j}{2}\right).
 \]

Finally, for every $i \in [1,n]$, considering the element of the corresponding $Y_i$ with $k=\frac{n-1}{2}$ we get
\[
\sum_{a_{i,j} \in R_i} a_{i,j} = \frac{n^2+n}{2}+(-1)^{i} \left( \frac{n^3+1}{2}-i\right).
\]
\end{rem}

\begin{ex}
Following the proof of Proposition \ref{prop:2n}, we obtain the $\N\H_{22}(11;11)$ below, whose elements belong to $\Z_{264}$:

\begin{center}
\begin{footnotesize}
$\begin{array}{|r|r|r|r|r|r|r|r|r|r|r|}\hline
1 & 119 & 25 & 95 & 49 & 71 & 73 & 47 & 97 & 23 & 121\\ \hline
130 & 14 & 106 & 38 & 82 & 62 & 58 & 86 & 34 & 110 & 10\\ \hline
3 & 117 & 27 & 93 & 51 & 69 & 75 & 45 & 99 & 21 & 123\\ \hline
128 & 16 & 104 & 40 & 80 & 64 & 56 & 88 & 32 & 112 & 8\\ \hline
5 & 115 & 29 & 91 & 53 & 67 & 77 & 43 & 101 & 19 & 125\\ \hline
126 & 18 & 102 & 42 & 78 & 66 & 54 & 90 & 30 & 114 & 6\\ \hline
7 & 113 & 31 & 89 & 55 & 65 & 79 & 41 & 103 & 17 & 127\\ \hline
124 & 20 & 100 & 44 & 76 & 68 & 52 & 92 & 28 & 116 & 4\\ \hline
9 & 111 & 33 & 87 & 57 & 63 & 81 & 39 & 105 & 15 & 129\\ \hline
122 & 22 & 98 & 46 & 74 & 70 & 50 & 94 & 26 & 118 & 2\\ \hline
11 & 109 & 35 & 85 & 59 & 61 & 83 & 37 & 107 & 13 & 131\\ \hline
\end{array}$
\end{footnotesize}
\end{center}
\end{ex}

\subsection{Globally simple $\N\H_{n^2}(n;n)$}

\begin{prop}\label{prop:n2}
For every odd integer $n\geq 1$, there exists a globally simple $\N\H_{n^2}(n;n)$.
\end{prop}
\begin{proof}
Let $n \geq 1$ be an odd integer, and consider the $n \times n$ matrix $A$ with elements in $\Z_{3n^2}$ whose $(i,j)$-th entry is:
\[
a_{i,j} = \left\{
\begin{aligned}
&\varepsilon\left(3 n (j-1)+3 (i-1)+1 \right) &\text{for  $j \in \left[ 1, \frac{n+1}{2} \right]$,} \\
&\varepsilon\left(3 n j-3 i+1 \right) &\text{for $j \in \left[\frac{n+3}{2},n\right]$,}\\
\end{aligned}
\right.
\]
where $\varepsilon = 1$ if $i \equiv j \pmod{2}$ and $\varepsilon = -1$ otherwise.
Condition $(\rm{a_1})$ of Definition \ref{def:NZS} is trivially verified. It is easy to see that for $j\in \left[1,\frac{n-1}{2}\right]$ :
$$supp(C_j \cup C_{n-j+1}) =[ 1+3n(j-1), 3nj-1 ]\setminus \{3k: k\in\left[n(j-1)+1,nj-1\right]\},$$

also,
\[
supp\left(C_{\frac{n+1}{2}}\right) = \left[1+3n\frac{n-1}{2}, \frac{3n^2-1}{2}\right] \setminus\left\{3k: k\in\left[ \frac{n^2-n+2}{2}, \frac{n^2-1}{2}\right]  \right\}.
\]
It then follows that condition $(\rm{b_1})$ of Definition \ref{def:NZS} is satisfied.

We now need to check that the property of being globally simple holds.
It is easy to see that for any odd $j \in \left[1,\frac{n+1}{2}\right]$, the $j$-th column of $A$ is:
\[
\begin{aligned}
C_j &= \big(3n(j-1)+1, -3n(j-1)-4, 3n(j-1)+7, -3n(j-1)-10, \dotsc,      \\
&-3n(j-1)-3(n-2)-1, 3n(j-1)+3(n-1)+1 \big).
\end{aligned}
\]

Define now:
\[
\begin{aligned}
X&= \left\{ -3k: k \in \left[1,\frac{n-1}{2} \right] \right\}, \\
Y_j &= \left\{3n(j-1) + 3k+1: k \in \left[0,\frac{n-1}{2}\right]\right\}. \\
\end{aligned}
\]
By applying Lemma \ref{lemma:main}(1), with $a =3n(j-1)+1$, $b =-3n(j-1)-4$, $g = 6$ and $\ell=\frac{n-1}{2}$, we obtain $S(C_j) = X \cup Y_j$.
It can then be seen that $X$ and $Y_j$ are sets non containing $0$. Note also that $X$ and $Y_j$ are trivially disjoint for every  odd $j\in \left[1,\frac{n+1}{2}\right]$, as their elements belong to different equivalence classes modulo $3$.

Now, for every odd $j\in \left[\frac{n+3}{2}, n\right]$, the $j$-th column of $A$ is:
$$ C_j = \big(3nj-2, -3nj+5, 3nj-8, -3nj+11, \dotsc,
-3nj+3(n-1)-1, 3nj-3n+1 \big). $$

Define now:
\[
\begin{aligned}
X&= \left\{ 3k : k \in \left[1,\frac{n-1}{2}\right] \right\}, \\
Y_j &= \left\{3nj-2-3k: k\in\left[ 0,\frac{n-1}{2}\right]\right\}. \\
\end{aligned}
\]
By applying Lemma \ref{lemma:main}(1) with $a=3nj-2$, $b=-3nj+5$, $g = -6$ and $\ell=\frac{n-1}{2}$, we obtain $S(C_j) = X \cup Y_j$.
Clearly, $X$ and $Y_j$ are sets, containing only non-zero elements. By following the same argument to the previous one, we have that they are disjoint for every odd $j\in \left[\frac{n+3}{2}, n\right] $.

It can be verified that for every even $j \in \left[1,\frac{n+1}{2}\right]$ the $j$-th column of $A$ is:
\[
\begin{aligned}
C_j &= \big(-3n(j-1)-1, 3n(j-1)+4, -3n(j-1)-7, 3n(j-1) + 10, \dotsc,\\
&3n(j-1)+3(n-2)+1,-3n(j-1) - 3(n-1)-1\big).
\end{aligned}
\]

Define:
\[
\begin{aligned}
X&= \left\{ 3k: k \in \left[1,\frac{n-1}{2}\right]\right\}, \\
Y_j &= \left\{-3n(j-1)-1 -3k: k \in \left[ 0,\frac{n-1}{2}\right] \right\}. \\
\end{aligned}
\]
By applying Lemma \ref{lemma:main}(1) with $a =-3n(j-1)-1$, $b =3n(j-1)+4$, $g =-6$ and $\ell=\frac{n-1}{2}$ it is easy to see that $S(C_j)= X \cup Y_j$. Moreover, $X$ and $Y_j$ are sets not containing zero that are disjoint for every even $j\in \left[1,\frac{n+1}{2}\right]$ since their elements belong to different equivalence classes modulo $3$.

Finally, for every even $j \in \left[ \frac{n+3}{2},n\right]$ the $j$-th column of $A$ is:
$$C_j = \big(-3nj+2, 3nj-5,-3nj +8, 3nj-11,\dotsc,3nj-3(n-1)+1, -3nj+3n-1\big).$$

Define then:
\[
\begin{aligned}
X&= \left\{ -3k: k \in \left[1,\frac{n-1}{2}\right] \right\}, \\
Y_j &= \left\{2-3nj +3k: k \in \left[ 0, \frac{n-1}{2}\right]\right\}. \\
\end{aligned}
\]
By applying Lemma \ref{lemma:main}(1) with $a =-3nj+2$, $b =3nj-5$, $g =6$ and $\ell=\frac{n-1}{2}$, we obtain $S(C_j) = X \cup Y_j$. As before, note that, $X$ and $Y_j$ are disjoint sets, containing only non-zero elements.

We have then proven that every column of the array admits a simple  natural ordering.

Now, in order to check that the same property holds also for the rows, we need to distinguish two cases according to the congruence class of $n$ modulo $4$.

\paragraph{\textbf{Case 1:}  $n \equiv 1 \pmod{4}$.\\ }

Let $i \in [1,n]$ be odd, it can then be checked that the $i$-th row of $A$ is:
\[
\tiny{
\begin{aligned}
R_i &= \Big( \underbrace{3i-2,-3n-3i+2, 3i-2+6n,-9n-3i+2, \dotsc, 3i-2+3n\frac{n-1}{2}}_{\frac{n+1}{2}\ \text{terms}}, \\
& \underbrace{3n\frac{n-3}{2}+3i-1,-3n\frac{n-5}{2}-3i+1,3n\frac{n-7}{2}+3i-1,-3n\frac{n-9}{2}-3i+1,  \dotsc,
3n+3i-1, -3i+1}_{\frac{n-1}{2}\ \text{terms}} \Big).
\end{aligned}}
\]

Define:
   \[
    \begin{aligned}
    X &= \left\{-3nk: \,k \in  \left[1,\frac{n-1}{4}\right] \right\}, \\
    Y_i &= \left\{3i-2+3nk:\, k \in \left[0,\frac{n-1}{4} \right]\right\}, \\
    Z_i &=\left\{3i-2+3n\left(k+\frac{n-1}{4}\right): \,k \in \left[ 1,\frac{n-1}{4}\right]\right\}, \\
    T_i &=\left\{ 6i - 3 + 3n\frac{3n-7}{4} -3n k :\, k\in\left[0,\frac{n-5}{4}\right] \right\}.   \\
    \end{aligned}
    \]
By applying Lemma \ref{lemma:main}(1) to the first $\frac{n+1}{2}$ elements of $R_i$, with $a = 3i-2$, $b = -3n -3i+2$, $g =6n $
and $\ell=\frac{n-1}{4}$, we obtain that $X \cup Y_i \subset S(R_i) $. Then, by applying Lemma \ref{lemma:main}(4) to the remaining $\frac{n-1}{2}$ elements of $R_i$ with $a = 1-3i$, $b = -1+3n+3i $, $g = -6n$ and $\ell=\frac{n-5}{4}$, and adding the $\frac{n+1}{2}$-th partial sum of $R_i$, that is $3i-2+3n\frac{n-1}{4}$, we obtain $S(R_i) = X \cup Y_i \cup Z_i \cup T_i$. It is easy to see that these are sets not containing zero; it then remains to check that they are pairwise disjoint.
Notice that some of these sets are trivially disjoint modulo $3n^2$, since their elements belong to different equivalence classes modulo $3$.
Hence, it suffices to check that $X \cap T_i = \emptyset$ and $Y_i \cap Z_i = \emptyset$. In the first case, assume by contradiction that there exist
$k_1 \in \left[1,\frac{n-1}{4}\right]$, $k_2 \in \left[0,\frac{n-5}{4}\right]$ and $i$ such that:
\[
-3nk_1 \equiv 6i-3+ 3n \frac{3n-7}{4} -3nk_2 \pmod{3n^2}.
\]
This can be rewritten as:
\[
6i-3+3n\left(\frac{3n-7}{4}+k_1-k_2 \right) \equiv 0 \pmod{3n^2}.
\]
By the hypothesis on $k_1$ and $k_2$, the left-hand side is strictly positive in $\Z$ and the
element in  brackets is less than or equal to $n-2$. This fact, combined with $i \leq n$, proves that the left-hand side in $\Z$ is an integer in $[1, 3n^2-3]$,
so we get a contradiction.
Now, if there exists $i$ such that $Y_i\cap Z_i \neq \emptyset$, we would have:
\[
3n\left(k_2-k_1 + \frac{n-1}{4} \right) \equiv 0 \pmod{3n^2}
\]
with $k_1 \in \left[0,\frac{n-1}{4}\right]$ and $k_2\in \left[1,\frac{n-1}{4}\right]$.
However, the term inside brackets is  strictly positive and cannot exceed $n/2-1$, thus the equation cannot hold.

It can be checked that for every even $i \in [1, n]$ the $i$-th row of $A$ is:
\[
\tiny{
\begin{aligned}
R_i &= \Big( \underbrace{-1-3(i-1), 1+3(i-1)+3n, -1-3(i-1)-6n, 1+3(i-1)+9n, \dotsc,-1- 3(i-1)-3n\frac{n-1}{2}}_{\frac{n+1}{2}\ \text{terms}}, \\
& \underbrace{1-3n\frac{n-3}{2}-3i, -1+3n\frac{n-5}{2}+3i,1-3n\frac{n-7}{2}-3i,-1+3n\frac{n-9}{2}+3i, \dotsc,-1+3n+3i, 1-3i}_{\frac{n-1}{2}\ \text{terms}} \Big).
\end{aligned}}
\]

Set:
   \[
    \begin{aligned}
    X &= \left\{3nk: \, k \in \left[1,\frac{n-1}{4}\right]\right\}, \\
Y_i &= \left\{2-3i-3nk:\,  k \in \left[0,\frac{n-1}{4}\right]\right\}, \\
Z_i &=\left\{2-3i-3n\left(k+\frac{n-1}{4}\right): \, k\in\left[1,\frac{n-1}{4}\right]\right\}, \\
T_i &=\left\{3 -6i+ 3n\frac{7-3n}{4}+3nk:\, k \in \left[ 0, \frac{n-5}{4}\right] \right\}.   \\
    \end{aligned}
    \]
By applying Lemma \ref{lemma:main}(1) to the first $\frac{n+1}{2}$ elements of $R_i$ with $a = -1-3(i-1)$, $b= 1+3(i-1)+3n$, $g=-6n$ and
$\ell=\frac{n-1}{4}$, we obtain $X \cup Y_i \subset S(R_i)$. Then, by applying Lemma \ref{lemma:main}(4) to the remaining $\frac{n-1}{2}$ terms of $R_i$, with $a = -1+3i$, $b = 1-3n-3i$, $g = 6n$ and $\ell=\frac{n-5}{4}$, and by adding the $\frac{n+1}{2}$-th partial sum of $R_i$ that is $2-3i-3n\frac{n-1}{4}$, we obtain $S(R_i) = X \cup Y_i \cup Z_i \cup T_i$. It can be seen that these are sets containing only non-zero elements, thus it remains to check that they are pairwise disjoint modulo $3n^2$.
Some pairs of these sets are trivially disjoint, as their elements belong to different equivalence classes modulo $3$.
We then need to verify $Y_i \cap Z_i = \emptyset$ and $X \cap T_i = \emptyset$. In the first case, suppose by contradiction that:
\[
3n \left(\frac{n-1}{4} + k_2-k_1 \right) \equiv 0 \pmod{3n^2}
\]
for some $k_1 \in \left[0,\frac{n-1}{4}\right]$ and $k_2 \in \left[1,\frac{n-1}{4}\right]$.
However, it can be seen that this cannot hold as $k_2-k_1 \leq \frac{n-1}{4}$. In the second case, again by contradiction,
suppose we have:
\[
3nk_1 \equiv 3-6i +3n \frac{7-3n}{4} + 3nk_2 \pmod{3n^2},
\]
for some $k_1 \in \left[1,\frac{n-1}{4}\right]$, $k_2 \in \left[0,\frac{n-5}{4}\right]$ and for some even $i$.
This is equivalent to:
\[
3-6i + 3n \left( \frac{7-3n}{4} + k_2-k_1 \right) \equiv 0 \pmod{3n^2}.
\]
Now, it can be seen that the term inside the brackets is strictly negative in $\Z$, and is bounded from below by $2-n$.
By noticing that $i < n$, it can be seen that the left-hand side is a negative integer greater than $3-3n^2$, so we get a contradiction.

\paragraph{\textbf{Case 2:}  $n \equiv 3 \pmod{4}$. \\}

It is not hard to see that, since we are working modulo $3n^2$, the $i$-th row of $A$, for every odd $i \in [1,n]$, is:
\[
\tiny{
\begin{aligned}
R_i &= \Big( \underbrace{1+3(i-1), -1-3(i-1)-3n, 1+3(i-1)+6n, -1-3(i-1)-9n, \dotsc,-1-3(i-1)-3n\frac{n-1}{2}}_{\frac{n+1}{2}\ \text{terms}}, \\
& \underbrace{1-3n\frac{n-3}{2}-3i, -1+3n\frac{n-5}{2}+3i, 1-3n\frac{n-7}{2}-3i, -1+3n\frac{n-9}{2}+3i, \dotsc,-1+3n+3i, 1-3i}_{\frac{n-1}{2}\ \text{terms}} \Big).  \\
\end{aligned}}
\]

Define:
   \[
\begin{aligned}
    X &= \left\{-3nk: \, k \in \left[1,\frac{n+1}{4}\right] \right\}, \\
Y_i &= \left\{3i-2+3nk:\, k \in \left[ 0,\frac{n-3}{4}\right] \right\}, \\
Z &=\left\{-3n\left(k+\frac{n+1}{4}\right): \, k \in \left[1,\frac{n-3}{4}\right]\right\}, \\
T_i &=\left\{ 1 - 3i - 3n \frac{3n-5}{4} +3nk:\,k \in \left[  0,\frac{n-3}{4}\right] \right\}.   \\
    \end{aligned}
    \]
Now, by applying Lemma \ref{lemma:main}(3) to the first $\frac{n+1}{2}$ elements of $R_i$, with $a = 3i-2$, $b = -3n -3i+2$, $g =6n $
and $\ell=\frac{n-3}{4}$ we obtain that $X \cup Y_i \subset S(R_i) $. Then, by applying Lemma \ref{lemma:main}(2) to the remaining $\frac{n-1}{2}$ elements with $a = 1-3i$, $b = -1+3n+3i $, $g = -6n$ and $\ell=\frac{n-3}{4}$, and adding the $\frac{n+1}{2}$-th partial sum of $R_i$ that is $-3n\frac{n+1}{4}$, we obtain $S(R_i) = X \cup Y_i \cup Z \cup T_i$. It can then be easily seen that these are sets not containing any zero element, thus it remains to check that they are pairwise disjoint.
By considering the same sets written modulo $3$, we only need to check that $X \cap Z= \emptyset$ and $Y_i \cap T_i = \emptyset$. In the first case suppose by contradiction that:
\[
-3n k_1 \equiv -3n \left(k_2 + \frac{n+1}{4} \right) \pmod{3n^2},
\]
for some $k_1 \in \left[1,\frac{n+1}{4}\right]$ and $k_2 \in \left[1,\frac{n-3}{4}\right]$.
Clearly, this cannot  hold since $k_2 - k_1 \leq \frac{n-7}{4}$.
Suppose now $Y_i \cap T_i \neq \emptyset$, namely that
\[
3i - 2 + 3nk_1 \equiv 1 - 3i- 3n \frac{3n-5}{4} + 3nk_2 \pmod{3n^2}
\]
for some $k_1, k_2 \in \left[0,\frac{n-3}{4}\right]$ and some odd $i$.
Then:
\[
3-6i -3n \left( \frac{3n-5}{4} -k_2+k_1 \right) \equiv 0 \pmod{3n^2}.
\]
It can then be seen that in $\Z$ the left-hand side is strictly negative, and that the term in brackets cannot exceed $n-2$.
Since $i \in[1,n]$ the left-hand side is an integer contained in $[-3n^2+3,-1]$, hence we have a contradiction.

Recalling that we are working in $\Z_{3n^2}$,
one can check that for every even $i \in [1, n]$  the $i$-th row of $A$ is:
\[\tiny{
\begin{aligned}
R_i &= \Big( \underbrace{-1-3(i-1), 1+3(i-1)+3n, -1-3(i-1)-6n, 1+3(i-1)+9n, \dotsc,1+3(i-1)+3n\frac{n-1}{2}}_{\frac{n+1}{2}\ \text{terms}}, \\
& \underbrace{-1+3n\frac{n-3}{2}+3i, 1-3n\frac{n-5}{2}-3i, -1+3n\frac{n-7}{2}+3i, 1-3n\frac{n-9}{2}-3i, \dotsc,1-3n-3i, -1+3i}_{\frac{n-1}{2}\ \text{terms}} \Big).
\end{aligned}}
\]
Define now:
   \[
\begin{aligned}
    X &= \left\{3nk: \, k \in \left[ 1,\frac{n+1}{4}\right] \right\}, \\
    Y_i &= \left\{2-3i-3nk:\, k \in \left[ 0,\frac{n-3}{4}\right] \right\}, \\
    Z &=\left\{3n\left(k+\frac{n+1}{4}\right): \, k \in \left[1,\frac{n-3}{4}\right]\right\},\\
    T_i &=\left\{-1+3i+3n \frac{3n-5}{4} -3nk:\,k \in \left[0,\frac{n-3}{4}\right] \right\}.   \\
    \end{aligned}
    \]
By applying Lemma \ref{lemma:main}(3) to the first $\frac{n+1}{2}$ elements of $R_i$ with $a = -1-3(i-1)$, $b= 1+3(i-1)+3n$, $g=-6n$ and
$\ell=\frac{n-3}{4}$, we obtain $X \cup Y_i\subset S(R_i)$. Finally, by applying Lemma \ref{lemma:main}(2) to the remaining $\frac{n-1}{2}$ terms, with $a = -1+3i$, $b = 1-3n-3i$, $g = 6n$ and $\ell=\frac{n-3}{4}$, and adding the $\frac{n+1}{2}$-th partial sum of $R_i$, that is $3n\frac{n+1}{4}$, we obtain $S(R_i) = X \cup Y_i \cup Z \cup T_i$.  We can see that these are sets containing only non-zero elements, so we only need to prove that they are pairwise disjoint.
As before, it suffices to verify $X \cap Z = \emptyset$ and $Y_i \cap T_i = \emptyset$. In the first case, by means of contradiction, suppose
\[
3n k_1 \equiv 3n \left(k_2 + \frac{n+1}{4}\right) \pmod{3n^2}
\]
for some $k_1\in \left[ 1,\frac{n+1}{4}\right]$ and $k_2\in \left[ 1,\frac{n-3}{4}\right]$.
This cannot hold since $k_2-k_1 \leq \frac{n-3}{4}$. In the second case, by contradiction,
assume that for some $k_1$, $k_2 \in \left[ 0,\frac{n-3}{4}\right]$ and even $i$ it holds:
\[
2-3i -3nk_1 \equiv -1+3i+ 3n \frac{3n-5}{4} - 3nk_2 \pmod{3n^2}.
\]
This is equivalent to:
\[
6i-3+3n\left( \frac{3n-5}{4} +k_1-k_2\right) \equiv 0 \pmod{3n^2}.
\]
It can be seen that the left-hand side is strictly positive in $\Z$; since $i\leq n$ and $k_1-k_2 \leq \frac{n-3}{4}$, the left-hand side is smaller than or equal to $3n^2-3$. Thus, the previous equation cannot hold, proving that $Y_i$ and $T_i$ are disjoint sets.

We then have shown that for every odd $n$ the rows and the columns of $A$ admit a simple  natural ordering.
This concludes the proof.
\end{proof}

\begin{rem}\label{rem:sum_n2}
For every $j \in [1,n]$ the sum of the elements of each column of the globally simple array $A$ described in Proposition \ref{prop:n2} is the following:
\[
\sum_{a_{i,j} \in C_j} a_{i,j}  =(-1)^{j-1} \left( 3nj-\frac{3n+1}{2} \right).
\]
Looking at the columns it is easy to see that for every $j \in [1,n]$ and for every $k \in [0,\lfloor\frac{n-5}{4}\rfloor]$ we have:
\[
			a_{1+2k,j}+a_{2+2k,j} = - (a_{n-2k,j} + a_{n-2k-1,j}).
\]
Hence if $n \equiv 1 \pmod{4}$ the total sum of the $j$-th column is given by $a_{\frac{n+1}{2},j}$, while for $n \equiv 3 \pmod{4}$ the sum is given by $a_{\frac{n-1}{2},j}+ a_{\frac{n+1}{2},j} + a_{\frac{n+3}{2},j}$.

Analogously, it can be seen that for every $i \in [1,n]$ and for every $k \in [0,\lfloor\frac{n-5}{4}\rfloor]$ we have:
\[
			a_{i,1+2k}+a_{i,2+2k} = - (a_{i,n-2k} + a_{i,n-2k-1}).
\]
We then obtain that for $n \equiv 1 \pmod{4}$ the total sum of the $i$-th row is given by $a_{i,\frac{n+1}{2}}$, while for $n \equiv 3 \pmod{4}$ the sum is given by $a_{i,\frac{n-1}{2}}+ a_{i,\frac{n+1}{2}} + a_{i,\frac{n+3}{2}}$.
Note also that $|a_{i,\frac{n-1}{2}} - a_{i,\frac{n+3}{2}}| = 1$.

Hence, for the rows we have that, if $n \equiv 1 \pmod{4}$, for every $i \in [1,n]$:
\[
\sum_{a_{i,j} \in R_i} a_{i,j} = (-1)^{i-1} \left(3i - 2 +3n \frac{n-1}{2}\right)=a_{i,\frac{n+1}{2}};
\]
 while, if  $n \equiv 3 \pmod{4}$, for every $i \in [1,n]$:
\[
\sum_{a_{i,j} \in R_i} a_{i,j} = (-1)^{i-1} \left(3(n+1-i) - 2 +3n \frac{n-1}{2}\right)=a_{i,\frac{n+1}{2}} + (-1)^i.
\]
\end{rem}

\begin{ex}
Following the proof of Proposition \ref{prop:n2}, we obtain the $\N\H_{169}(13;13)$ below, whose elements belong to $\Z_{507}$:

\begin{center}
\begin{footnotesize}
$\begin{array}{|r|r|r|r|r|r|r|r|r|r|r|r|r|}\hline
1 & -40 & 79 & -118 & 157 & -196 & 235 & 197 & -158 & 119 & -80 & 41 & -2\\ \hline
-4 & 43 & -82 & 121 & -160 & 199 & -238 & -200 & 161 & -122 & 83 & -44 & 5\\ \hline
7 & -46 & 85 & -124 & 163 & -202 & 241 & 203 & -164 & 125 & -86 & 47 & -8\\ \hline
-10 & 49 & -88 & 127 & -166 & 205 & -244 & -206 & 167 & -128 & 89 & -50 & 11\\ \hline
13 & -52 & 91 & -130 & 169 & -208 & 247 & 209 & -170 & 131 & -92 & 53 & -14\\ \hline
-16 & 55 & -94 & 133 & -172 & 211 & -250 & -212 & 173 & -134 & 95 & -56 & 17\\ \hline
19 & -58 & 97 & -136 & 175 & -214 & 253 & 215 & -176 & 137 & -98 & 59 & -20\\ \hline
-22 & 61 & -100 & 139 & -178 & 217 & 251 & -218 & 179 & -140 & 101 & -62 & 23\\ \hline
25 & -64 & 103 & -142 & 181 & -220 & -248 & 221 & -182 & 143 & -104 & 65 & -26\\ \hline
-28 & 67 & -106 & 145 & -184 & 223 & 245 & -224 & 185 & -146 & 107 & -68 & 29\\ \hline
31 & -70 & 109 & -148 & 187 & -226 & -242 & 227 & -188 & 149 & -110 & 71 & -32\\ \hline
-34 & 73 & -112 & 151 & -190 & 229 & 239 & -230 & 191 & -152 & 113 & -74 & 35\\ \hline
37 & -76 & 115 & -154 & 193 & -232 & -236 & 233 & -194 & 155 & -116 & 77 & -38\\ \hline
\end{array}$
\end{footnotesize}
\end{center}

And similarly, we obtain the $\N\H_{121}(11;11)$ below, filled with elements in $\Z_{363}$:

\begin{center}
\begin{footnotesize}
$\begin{array}{|r|r|r|r|r|r|r|r|r|r|r|r|r|}\hline
1 & -34 & 67 & -100 & 133 & -166 & -134 & 101 & -68 & 35 & -2\\ \hline
-4 & 37 & -70 & 103 & -136 & 169 & 137 & -104 & 71 & -38 & 5\\ \hline
7 & -40 & 73 & -106 & 139 & -172 & -140 & 107 & -74 & 41 & -8\\ \hline
-10 & 43 & -76 & 109 & -142 & 175 & 143 & -110 & 77 & -44 & 11\\ \hline
13 & -46 & 79 & -112 & 145 & -178 & -146 & 113 & -80 & 47 & -14\\ \hline
-16 & 49 & -82 & 115 & -148 & 181 & 149 & -116 & 83 & -50 & 17\\ \hline
19 & -52 & 85 & -118 & 151 & 179 & -152 & 119 & -86 & 53 & -20\\ \hline
-22 & 55 & -88 & 121 & -154 & -176 & 155 & -122 & 89 & -56 & 23\\ \hline
25 & -58 & 91 & -124 & 157 & 173 & -158 & 125 & -92 & 59 & -26\\ \hline
-28 & 61 & -94 & 127 & -160 & -170 & 161 & -128 & 95 & -62 & 29\\ \hline
31 & -64 & 97 & -130 & 163 & 167 & -164 & 131 & -98 & 65 & -32\\ \hline
\end{array}$
\end{footnotesize}
\end{center}

\end{ex}

\subsection{Globally simple $\N\H_{2n^2}(n;n)$}

\begin{prop}\label{prop:2n2}
For every  odd integer $n\geq 1$, there exists a globally simple $\N\H_{2n^2}(n;n)$.
\end{prop}
\begin{proof}
Let $n\geq 1$ be odd, and consider the $n \times n$ matrix $A$ with elements in $\Z_{4n^2}$ whose $(i,j)$-th entry is:
\[
a_{i,j} = \left\{
\begin{aligned}
&2n(j-1)+2i-1 &\text{for $i \equiv j \pmod{2}$ and $j \in \left[ 1, \frac{n+1}{2}\right]$,} \\
&2 n(n-j+1)-2i+1 &\text{for $i \not\equiv j \pmod{2}$ and $j \in \left[1, \frac{n+1}{2}\right]$,} \\
&2nj-2i+1 &\text{for $i \equiv j \pmod{2}$ and $j \in\left[\frac{n+3}{2},n \right]$,}\\
&2 n(n-j)+2 i-1  &\text{for $i \not\equiv j \pmod{2}$ and $j \in \left[\frac{n+3}{2},n\right]$.}\\
\end{aligned}
\right.
\]
Since $A$ has no empty cells, condition  $(\rm{a_1})$ of Definition \ref{def:NZS} is trivially satisfied.

For $j \in \left[1,\frac{n-1}{2}\right]$, one can check that
$supp(C_j \cup C_{n-j+1})$ contains exactly all the odd elements of $[1+2n(j-1), 2nj-1] \cup [2n^2+1-2nj, 2n^2-1-2n(j-1)]$,
and that  $supp\left(C_{\frac{n+1}{2}}\right) $ contains exactly all the odd elements of $[n(n-1)+1, n(n+1)-1] $.
Since we have to avoid the elements of the subgroup of order $2n^2$ of $\Z_{4n^2}$,
condition $(\rm{b_1})$ of Definition \ref{def:NZS} is satisfied too.

Now we consider the partial sums of the columns of $A$.
By the description of $A$ it immediately follows  that for every odd  $j \in \left[1,\frac{n+1}{2}\right]$ we have:
\[
\begin{aligned}
C_j &= \big(2n(j-1)+1, 2n(n-j+1)-3, 2n(j-1)+5, 2n(n-j+1)-7, \dotsc, \\
& 2n(n-j+1)-2n+1, 2n(j-1)+2n-1 \big).
\end{aligned}
\]

Define now:
\[
\begin{aligned}
X&= \left\{ k (2n^2 - 2): k \in \left[1,\frac{n-1}{2}\right]\right\}, \\
Y_j &= \left\{1 + 2n(j-1) + k(2n^2+2): k\in\left[0,\frac{n-1}{2}\right]\right\}. \\
\end{aligned}
\]
By applying Lemma \ref{lemma:main}(1) with $a =2n(j-1)+1 $, $b = 2n(n-j+1)-3$, $g=4$
and $\ell=\frac{n-1}{2}$, we obtain $S(C_j) = X   \cup Y_j$.
It can be checked that both $X$ and $Y_j$ are sets, containing no zero elements, and they are trivially disjoint as their elements belong to different classes modulo $2$.
Similarly, it can be seen that for every odd $j\in \left[ \frac{n+3}{2}, n \right]$, the $j$-th column of $A$ is:
$$C_j= \big(2nj-1, 2n(n-j)+3,2nj-5,2n(n-j)+7, \dotsc,2n(n-j)+2n-3, 2nj-2n+1\big).$$

Set:
\[
\begin{aligned}
X&= \left\{ k (2n^2 + 2): k \in \left[1,\frac{n-1}{2}\right]\right\}, \\
Y_j &= \left\{2nj-1+ k(2n^2-2): k\in\left[0,\frac{n-1}{2}\right]\right\}. \\
\end{aligned}
\]
By applying Lemma \ref{lemma:main}(1) with $a = 2nj-1$, $b=2n(n-j)+3$, $g=-4$ and $\ell=\frac{n-1}{2}$,
we get $S(C_j) = X \cup Y_j$. It can then be easily verified that these are disjoint sets non containing zero.

Now, for any even $j \in \left[1,\frac{n+1}{2}\right]$  the $j$-th column of $A$ is:
\[
\begin{aligned}
C_j &= \big(2n(n-j+1)-1, 2n(j-1)+3, 2n(n-j+1)-5,2n(j-1)+7, \dotsc, \\
    &2n(j-1)+2n-3 , 2n(n-j+1)-2n+1\big).  \\
\end{aligned}
\]

Consider:
\[
\begin{aligned}
X&= \left\{ k (2n^2 + 2): k \in \left[ 1,\frac{n-1}{2}\right]\right\}, \\
Y_j &= \left\{2n(n-j+1)-1 + k(2n^2-2):k \in \left[0,\frac{n-1}{2}\right]\right\}. \\
\end{aligned}
\]
By applying Lemma \ref{lemma:main}(1) with $a=2n(n-j+1)-1$, $b=2n(j-1)+3 $, $g=-4$ and $\ell=\frac{n-1}{2}$ we have $S(C_j) = X \cup Y_j$. It can be seen that $X$ and $Y_j$ are sets, containing exclusively non-zero elements, and as before, $X$ and $Y_j$ are trivially disjoint.

Similarly, it can be seen that for every even $j \in \left[\frac{n+3}{2},n\right]$ the $j$-th column of $A$ is:
$$C_j = \big(2n(n-j)+1,2nj-3,2n(n-j)+5,2nj-7, \dotsc, 2nj-2n+3, 2n(n-j)+2n-1\big).$$

By applying Lemma \ref{lemma:main}(1) with $a=2n(n-j)+1$, $b=2nj-3$, $g=4$ and $\ell=\frac{n-1}{2}$, we have $S(C_j) = X \cup Y_j$, where:
\[
\begin{aligned}
X&= \left\{ k (2n^2 - 2): k \in \left[ 1, \frac{n-1}{2}\right]\right\}, \\
Y_j &= \left\{2n(n-j)+1 + k(2n^2+2): k \in \left[0,\frac{n-1}{2}\right]\right\}, \\
\end{aligned}
\]
which can be checked to be disjoint sets not containing zero elements.

It remains to consider the partial sums of every row. We have to split the proof into
two cases according to the congruence class of $n$ modulo $4$.
\paragraph{\textbf{Case 1:}  $n \equiv 1 \pmod{4}$.\\ }
For every odd $i \in [1,n]$, the $i$-th row of $A$ is:
\[
\tiny{
\begin{aligned}
R_i &= \Big( \underbrace{2i-1, 2n(n-1)-2i+1,4n+2i-1,2n(n-3)-2i+1, \dotsc,2n\frac{n-1}{2}+2i-1}_{\frac{n+1}{2}\ \text{terms}}, \\
& \underbrace{2n\frac{n-3}{2}+2i-1, 2n\frac{n+5}{2}-2i+1, 2n\frac{n-7}{2}+2i-1, 2n\frac{n+9}{2}-2i+1, \dotsc, 2n+2i-1, 2n^2-2i+1}_{\frac{n-1}{2}\ \text{terms}} \Big).  \\
\end{aligned}}
\]
Set:
   \[
\begin{aligned}
    X &= \left\{k(2n^2-2n): \, k \in \left[1,\frac{n-1}{4}\right]\right\}, \\
    Y_i &= \left\{2i-1+k(2n^2+2n):\,  k \in \left[0,\frac{n-1}{4}\right] \right\}, \\
    Z_i &=\left\{-1+2i+(2n^2+2n)\left(k+\frac{n-1}{4}\right): \, k\in\left[1,\frac{n-1}{4}\right]\right\}, \\
		T_i &=\left\{4i-2+\frac{n^3+2n^2-7n}{2}+k(2n^2-2n):\, k \in \left[ 0,\frac{n-5}{4}\right] \right\}.   \\
    \end{aligned}
    \]
By applying Lemma \ref{lemma:main}(1) to the first $\frac{n+1}{2}$ elements of $R_i$ with $a =2i-1$, $b =2n(n-1)-2i+1 $, $g=4n$ and $\ell=\frac{n-1}{4}$, we obtain $X \cup Y_i \subset S(R_i)$. Then, by applying Lemma \ref{lemma:main}(4) to the remaining elements with $a =2n^2-2i+1$,
$b = 2n+2i-1$, $g=-4n$ and $\ell=\frac{n-5}{4}$, and by adding the  $\frac{n+1}{2}$-th partial sum of $R_i$,
that is $2i-1+(2n^2+2n)\frac{n-1}{4}$, we obtain $S(R_i) = X \cup Y_i \cup Z_i \cup T_i$.
We can see that these are sets, containing only non-zero elements. It remains to check that they are pairwise disjoint.
By checking these sets modulo $2$, it is easy to see that it suffices to verify that $Y_i \cap Z_i = \emptyset$ and $X \cap T_i = \emptyset$.
In the first case, assuming that there exists an element in the intersection, this would imply that for some
$k_1 \in \left[0,\frac{n-1}{4}\right]$, $k_2 \in \left[1,\frac{n-1}{4}\right]$ and $i$ odd the following holds:
\[
2i-1+k_1(2n^2+2n) \equiv 2i-1+(2n^2+2n)\left(k_2+\frac{n-1}{4}\right) \pmod{4n^2},
\]
that is:
\[
\left(k_2+\frac{n-1}{4}-k_1\right)(2n^2+2n) \equiv 0 \pmod{4n^2}.
\]
Since we are working in $\Z_{4n^2}$, $2n^2(k_2+\frac{n-1}{4}-k_1) \in \{0,2n^2\}$, we get a contradiction by noticing that 
$k_2+\frac{n-1}{4}-k_1 \leq \frac{n-1}{2}$.

In the second case, again by contradiction, suppose we have:
\[
k_1 (2n^2-2n) \equiv 4i-2+\frac{n^3+2n^2-7n}{2}+k_2(2n^2-2n)\pmod{4n^2},
\]
for some $k_1\in \left[1,\frac{n-1}{4}\right]$, $k_2 \in \left[0,\frac{n-5}{4}\right]$ and some odd $i$. This can be rewritten as:
\[
-2+4i + \frac{n^3+2n^2-7n}{2}+(k_2-k_1)(2n^2-2n) \equiv 0 \pmod{4n^2}.
\]
Note that, since $n\equiv 1\pmod{4}$, we have that $n^3 =(4l+1) n^2 \equiv n^2 \pmod{4n^2}$, hence we can  rewrite the previous expression modulo $2n^2$, obtaining:
\[
-2+4i + \frac{3n^2-7n}{2}-2n(k_2-k_1)\equiv 0 \pmod{2n^2},
\]
which is equivalent to:
\[
-2+4i+ 2n \left(\frac{3n-7}{4}-k_2+k_1 \right) \equiv 0 \pmod{2n^2}.
\]
After noticing that the left-hand side is strictly positive in $\Z$, we observe that the bracket is bounded from above by $n-2$,
thus from $i \leq n$, the left-hand side seen in $\Z$ is less than or equal to $2n^2-2$, that is a contradiction.

For every even $i \in [1,n]$  the $i$-th row of $A$ is:
\[
\tiny{
\begin{aligned}
R_i &= \Big( \underbrace{2n^2-2i+1,2n+2i-1,2n(n-2)-2i+1,6n+2i-1, \dotsc,2n\frac{n+1}{2}-2i+1}_{\frac{n+1}{2}\ \text{terms}}, \\
& \underbrace{2n\frac{n+3}{2}-2i+1, 2n\frac{n-5}{2}+2i-1, 2n\frac{n+7}{2}-2i+1, 2n\frac{n-9}{2}+2i-1, \dotsc, 2n(n-1)-2i+1, 2i-1}_{\frac{n-1}{2}\ \text{terms}} \Big).
\end{aligned}}
\]
Let now:
   \[
\begin{aligned}
    X &= \left\{k(2n^2+2n): \, k \in \left[1,\frac{n-1}{4}\right]\right\}, \\
Y_i &= \left\{2n^2+1-2i+k(2n^2-2n):\, k \in \left[ 0,\frac{n-1}{4}\right] \right\}, \\
Z_i &=\left\{2n^2+1-2i+(2n^2-2n)\left(\frac{n-1}{4}+k\right):\,k \in \left[ 1,\frac{n-1}{4}\right]\right\}, \\
T_i &=\left\{ 2-4i+\frac{n^3+4n^2+7n}{2}  + k(2n^2+2n):\,  k \in \left[0,\frac{n-5}{4}\right] \right\}.   \\
    \end{aligned}
    \]
By applying Lemma \ref{lemma:main}(1) to the first $\frac{n+1}{2}$ elements of $R_i$ with $a = 2n^2-2i+1$, $b =2n+2i-1$,
$g=-4n$ and $\ell=\frac{n-1}{4}$, we obtain $X \cup Y_i \subset S(R_i)$.
Then, by applying Lemma \ref{lemma:main}(4) to the remaining elements with $a =2i-1$, $b =2n(n-1)-2i+1$, $g=4n$
and $\ell=\frac{n-5}{4}$, and by adding the $\frac{n+1}{2}$-th partial sum of $R_i$, that is $2n^2+1-2i+(2n^2-2n)\frac{n-1}{4}$, we obtain $S(R_i) = X \cup Y_i \cup Z_i \cup T_i$.
It can be easily seen that these are sets containing only non-zero elements, thus it only remains to check that they have trivial intersection.
As before, it suffices to check that $Y_i \cap Z_i = \emptyset$ and that $X \cap T_i = \emptyset$.
In the first case, by contradiction, we would have for some $k_1\in \left[ 0,\frac{n-1}{4}\right]$, $k_2 \in \left[ 1,\frac{n-1}{4}\right]$ and even $i$:
\[
2n^2+1-2i+k_1(2n^2-2n) \equiv 2n^2+1-2i+(2n^2-2n)\left(\frac{n-1}{4}+k_2\right) \pmod{4n^2},
\]
that is:
\[
(2n^2-2n)\left(\frac{n-1}{4}+k_2 -k_1\right)  \equiv 0 \pmod{4n^2}.
\]
We then have that $2n^2(\frac{n-1}{4}+k_2 -k_1) \in \{0,2n^2\}$ modulo $4n^2$, thus we obtain a contradiction by noticing that $\frac{n-1}{4}+k_2 -k_1 $ is strictly positive and bounded by $\frac{n-1}{2}$.

In the second case, again by contradiction, we would have for some $k_1 \in \left[1,\frac{n-1}{4}\right]$,
$k_2 \in\left[0,\frac{n-5}{4}\right]$ and even $i$ that:
\[
k_1(2n^2+2n) \equiv  2-4i+\frac{n^3+4n^2+7n}{2} + k_2(2n^2+2n) \pmod{4n^2}.
\]
As before $n^3\equiv n^2 \pmod{4n^2}$, hence we would have:
\[
2-4i+\frac{n^2+7n}{2} + 2n(k_2-k_1) \equiv 0 \pmod{2n^2}.
\]
It can be seen that the left-hand side is strictly positive in $\Z$, and as $i \in[1, n] $ and $k_2-k_1 \in\left[\frac{1-n}{4},\frac{n-9}{4}\right]$, is bounded from below by $2$ and from above by $n^2-n-2$,
hence we have a contradiction.

\paragraph{\textbf{Case 2:}  $n \equiv 3 \pmod{4}$.\\ }

It can be checked that for every odd $i \in [1,n]$ the $i$-th row of $A$ is:
\[\tiny{
\begin{aligned}
R_i &= \Big( \underbrace{2i-1, 2n(n-1)-2i+1,4n+2i-1,2n(n-3)-2i+1, \dotsc,2n\frac{n+1}{2}-2i+1}_{\frac{n+1}{2}\ \text{terms}}, \\
& \underbrace{2n\frac{n+3}{2}-2i+1, 2n\frac{n-5}{2}+2i-1, 2n\frac{n+7}{2}-2i+1, 2n\frac{n-9}{2}+2i-1, \dotsc, 2n+2i-1, 2n^2-2i+1}_{\frac{n-1}{2}\ \text{terms}} \Big).
\end{aligned}}
\]
Now, define:
\[
    \begin{aligned}
    X &= \left\{k(2n^2-2n): \, k \in \left[1,\frac{n+1}{4}\right]\right\}, \\
Y_i &= \left\{2i-1+k(2n^2+2n):\,  k \in \left[0,\frac{n-3}{4}\right]\right\}, \\
Z &=\left\{(2n^2-2n)\left(k+\frac{n+1}{4}\right): \, k\in\left[1,\frac{n-3}{4}\right]\right\}, \\
T_i &=\left\{1-2i+\frac{n^3+2n^2+5n}{2}+k(2n^2+2n):\, k \in \left[ 0, \frac{n-3}{4}\right] \right\}.   \\
    \end{aligned}
\]
As done for the case $n \equiv 1 \pmod{4}$, it is not hard to see that these are disjoint sets non containing zero.
Also, applying again Lemma \ref{lemma:main}, one can check that  $S(R_i) = X \cup Y_i \cup Z \cup T_i$.

For every even $i \in [1,n]$ the $i$-th row of $A$ is:
\[\tiny{
\begin{aligned}
R_i &= \Big( \underbrace{2n^2-2i+1,2n+2i-1,2n(n-2)-2i+1,6n+2i-1, \dotsc,2n\frac{n-1}{2}+2i-1}_{\frac{n+1}{2}\ \text{terms}}, \\
& \underbrace{2n\frac{n-3}{2}+2i-1, 2n\frac{n+5}{2}-2i+1, 2n\frac{n-7}{2}+2i-1, 2n\frac{n+9}{2}-2i+1, \dotsc, 2n(n-1)-2i+1, 2i-1}_{\frac{n-1}{2}\ \text{terms}} \Big).
\end{aligned}}
\]
Let now:

\[
    \begin{aligned}
    X &= \left\{k(2n^2+2n): \, k \in \left[1,\frac{n+1}{4}\right]\right\}, \\
Y_i &= \left\{2n^2-2i+1+k(2n^2-2n):\,  k \in \left[0,\frac{n-3}{4}\right]\right\}, \\
Z &=\left\{(2n^2+2n)\left(k+\frac{n+1}{4}\right): \, k\in\left[1,\frac{n-3}{4}\right]\right\}, \\
T_i &=\left\{2i-1+\frac{n^3+4n^2-5n}{2}+k(2n^2-2n):\, k \in \left[ 0, \frac{n-3}{4}\right] \right\}.   \\
    \end{aligned}
\]
It can then be seen that these are pairwise disjoint sets not containing zero,  and,
applying Lemma \ref{lemma:main}, that $S(R_i) = X \cup Y_i \cup Z \cup T_i$.

We can then conclude that for every odd $n$ the array $A$ is a globally simple $\N\H_{2n^2}(n;n)$.
\end{proof}

\begin{rem}\label{rem:sum_2n2}
For every odd integer $n$ we write the total sum of every  row and column of the array $A$ constructed in the proof of Proposition \ref{prop:2n2}. In particular, for every $j \in [1,n]$, if we take the element of the corresponding set $Y_j$ with $k=\frac{n-1}{2}$, we get:
\[
\sum_{a_{i,j} \in C_j} a_{i,j} =n^3 + (-1)^{j-1} (2nj-n^2-n).
\]
Analogously, if $n \equiv 1 \pmod{4}$ then  for every $i \in [1,n]$ if we take the element of the corresponding set $T_i$ with $k=\frac{n-5}{4}$, we get:
\[
\sum_{a_{i,j} \in R_i} a_{i,j} =n^3 + (-1)^{i-1} (2i-n-1);
\]
if $n \equiv 3 \pmod{4}$ then  for every $i \in [1,n]$ taking the element of the corresponding set $T_i$ with $k=\frac{n-3}{4}$, we obtain:
\[
\sum_{a_{i,j} \in R_i} a_{i,j} =n^3 + (-1)^{i-1} \left(n+1-2i\right).
\]
\end{rem}

\begin{ex}
Following the proof of Proposition \ref{prop:2n2}, we obtain the $\N\H_{338}(13;13)$ below, whose elements belong to $\Z_{676}$:
\begin{center}
\begin{footnotesize}
$\begin{array}{|r|r|r|r|r|r|r|r|r|r|r|r|r|}\hline
1 & 311 & 53 & 259 & 105 & 207 & 157 & 131 & 233 & 79 & 285 & 27 & 337\\ \hline
335 & 29 & 283 & 81 & 231 & 133 & 179 & 205 & 107 & 257 & 55 & 309 & 3\\ \hline
5 & 307 & 57 & 255 & 109 & 203 & 161 & 135 & 229 & 83 & 281 & 31 & 333\\ \hline
331 & 33 & 279 & 85 & 227 & 137 & 175 & 201 & 111 & 253 & 59 & 305 & 7\\ \hline
9 & 303 & 61 & 251 & 113 & 199 & 165 & 139 & 225 & 87 & 277 & 35 & 329\\ \hline
327 & 37 & 275 & 89 & 223 & 141 & 171 & 197 & 115 & 249 & 63 & 301 & 11\\ \hline
13 & 299 & 65 & 247 & 117 & 195 & 169 & 143 & 221 & 91 & 273 & 39 & 325\\ \hline
323 & 41 & 271 & 93 & 219 & 145 & 167 & 193 & 119 & 245 & 67 & 297 & 15\\ \hline
17 & 295 & 69 & 243 & 121 & 191 & 173 & 147 & 217 & 95 & 269 & 43 & 321\\ \hline
319 & 45 & 267 & 97 & 215 & 149 & 163 & 189 & 123 & 241 & 71 & 293 & 19\\ \hline
21 & 291 & 73 & 239 & 125 & 187 & 177 & 151 & 213 & 99 & 265 & 47 & 317\\ \hline
315 & 49 & 263 & 101 & 211 & 153 & 159 & 185 & 127 & 237 & 75 & 289 & 23\\ \hline
25 & 287 & 77 & 235 & 129 & 183 & 181 & 155 & 209 & 103 & 261 & 51 & 313\\ \hline
\end{array}$
\end{footnotesize}
\end{center}

And similarly, we obtain the $\N\H_{450}(15;15)$ below, filled with elements in $\Z_{900}$:

\begin{center}
\begin{footnotesize}
$\begin{array}{|r|r|r|r|r|r|r|r|r|r|r|r|r|r|r|}\hline
1 & 419 & 61 & 359 & 121 & 299 & 181 & 239 & 269 & 151 & 329 & 91 & 389 & 31 & 449\\ \hline
447 & 33 & 387 & 93 & 327 & 153 & 267 & 213 & 183 & 297 & 123 & 357 & 63 & 417 & 3\\ \hline
5 & 415 & 65 & 355 & 125 & 295 & 185 & 235 & 265 & 155 & 325 & 95 & 385 & 35 & 445\\ \hline
443 & 37 & 383 & 97 & 323 & 157 & 263 & 217 & 187 & 293 & 127 & 353 & 67 & 413 & 7\\ \hline
9 & 411 & 69 & 351 & 129 & 291 & 189 & 231 & 261 & 159 & 321 & 99 & 381 & 39 & 441\\ \hline
439 & 41 & 379 & 101 & 319 & 161 & 259 & 221 & 191 & 289 & 131 & 349 & 71 & 409 & 11\\ \hline
13 & 407 & 73 & 347 & 133 & 287 & 193 & 227 & 257 & 163 & 317 & 103 & 377 & 43 & 437\\ \hline
435 & 45 & 375 & 105 & 315 & 165 & 255 & 225 & 195 & 285 & 135 & 345 & 75 & 405 & 15\\ \hline
17 & 403 & 77 & 343 & 137 & 283 & 197 & 223 & 253 & 167 & 313 & 107 & 373 & 47 & 433\\ \hline
431 & 49 & 371 & 109 & 311 & 169 & 251 & 229 & 199 & 281 & 139 & 341 & 79 & 401 & 19\\ \hline
21 & 399 & 81 & 339 & 141 & 279 & 201 & 219 & 249 & 171 & 309 & 111 & 369 & 51 & 429\\ \hline
427 & 53 & 367 & 113 & 307 & 173 & 247 & 233 & 203 & 277 & 143 & 337 & 83 & 397 & 23\\ \hline
25 & 395 & 85 & 335 & 145 & 275 & 205 & 215 & 245 & 175 & 305 & 115 & 365 & 55 & 425\\ \hline
423 & 57 & 363 & 117 & 303 & 177 & 243 & 237 & 207 & 273 & 147 & 333 & 87 & 393 & 27\\ \hline
29 & 391 & 89 & 331 & 149 & 271 & 209 & 211 & 241 & 179 & 301 & 119 & 361 & 59 & 421\\ \hline
\end{array}$
\end{footnotesize}
\end{center}
\end{ex}

\subsection{Globally simple $\N\H_{t}(n;n)$, $t$ divides $n$}
\begin{thm}\label{thm:t}
For every odd integer $n\geq 1$ and for every divisor $t$ of $n$, there exists a globally simple $\N\H_t(n;n)$.
\end{thm}
\begin{proof}
Let $n$ and $t$ be as in the statement and set $v=2n^2+t$.
Let $H$ be the $n \times t$ array with elements in $\Z_v$ whose columns are so defined:
for $j$ odd, 
set
\begin{eqnarray}
\nonumber C_j &=& \left((j-1)\frac{v}{t}+1, j\frac{v}{t}-2, (j-1)\frac{v}{t}+3, j\frac{v}{t}-4,\ldots, j\frac{v}{t}-(n-1),\right. \\
\nonumber  & & \left. (j-1)\frac{v}{t}+n\right),
\end{eqnarray}
for $j$ even, 
set
\begin{eqnarray}
\nonumber C_j &=& \left(-(j-1)\frac{v}{t}-1, -j\frac{v}{t}+2, -(j-1)\frac{v}{t}-3, -j\frac{v}{t}+4,\ldots, -j\frac{v}{t}+(n-1),\right. \\
\nonumber  & & \left. -(j-1)\frac{v}{t}-n\right).
\end{eqnarray}

For $\alpha \in \left[1, \frac{n}{t}-1\right]$, let $H_{\alpha n}$ be the matrix obtained from $H$ by adding
$\varepsilon \alpha n$ to the elements of $H$, where $\varepsilon =1$ for the elements of $H$ in the positions $(i,j)$
with $i \equiv j\pmod 2$, $\varepsilon =-1$ for the elements of $H$ in the positions $(i,j)$
with $i \not\equiv j\pmod 2$.
Now let $A$ be $n \times n$ array so constructed:
$$\begin{array}{|r|r|r|r|r|r|r|}\hline
H & -H_{n} & H_{2n} & -H_{3n} & H_{4n} & \ldots & H_{(\frac{n}{t}-1)n}\\ \hline
\end{array}
$$
where if $t=n$, $A=H$ is understood.
To prove that $supp(A)$ is the required one, it is useful to observe that, since the elements belong to $\Z_v$, some columns of $H$ can be written also in the following 
equivalent expression: 
if $j$ odd and $j\in \left[\frac{t+3}{2},t\right]$ then 
\begin{eqnarray}
\nonumber C_j &=& \left((j-t-1)\frac{v}{t}+1,(j-t)\frac{v}{t}-2, (j-t-1)\frac{v}{t}+3, (j-t)\frac{v}{t}-4,\ldots, \right. \\
\nonumber  & & \left. (j-t)\frac{v}{t}-(n-1),(j-t-1)\frac{v}{t}+n\right),
\end{eqnarray}
while if $j$ even and $j\in \left[\frac{t+3}{2},t\right]$ then
\begin{eqnarray}
\nonumber C_j &=& \left((t-j+1)\frac{v}{t}-1,(t-j)\frac{v}{t}+2, (t-j+1)\frac{v}{t}-3, (t-j)\frac{v}{t}+4,\ldots, \right. \\
\nonumber  & & \left. (t-j)\frac{v}{t}+(n-1),(t-j+1)\frac{v}{t}-n\right).
\end{eqnarray}
Now, it is not hard to see that for any $j\in \left[0, \frac{t-3}{2} \right]$ it results
$$supp(C_{j+1} \cup C_{t-j})=\left[j \frac{v}{t}+1, j\frac{v}{t}+n\right]\cup \left[(j+1) \frac{v}{t}-n, (j+1)\frac{v}{t}-1\right]$$
also
$$supp\left(C_{\frac{t+1}{2}}\right)=\left[ \frac{t-1}{2}\cdot \frac{v}{t}+1, \frac{t-1}{2}\cdot \frac{v}{t}+n\right].$$
By the definition of the arrays $H_{\alpha n}$, it follows that
$$supp(H_{\alpha n})=\left(\left[j \frac{v}{t}+1, j\frac{v}{t}+n\right]+\alpha n\right) \cup
\left(\left[(j+1) \frac{v}{t}-n, (j+1)\frac{v}{t}-1\right] -\alpha n\right) \cup$$
$$\cup\left(\left[ \frac{t-1}{2}\cdot \frac{v}{t}+1, \frac{t-1}{2}\cdot \frac{v}{t}+n\right] + \alpha n\right)$$
where $j\in \left[0, \frac{t-3}{2} \right]$.
This implies that $supp(A)=\left[1, \frac{v-1}{2} \right]\setminus\left\{\frac{v}{t},2\frac{v}{t}, \ldots, \frac{t-1}{2}\frac{v}{t}\right\}$.

Now we consider the partial sums of the columns of $A$. If for every odd $j \in [1,t]$ we apply Lemma \ref{lemma:main}(1) to the $j$-th column of $A$,
with $a = (j-1)\frac{v}{t} + 1 $, $b = j\frac{v}{t}-2$, $g = 2$ and $\ell = \frac{n-1}{2}$, we get $S(C_j) = X_j \cup Y_j$, where:
\[
\begin{aligned}
X_j &= \left\{ k\left[(2j-1)\frac{v}{t}-1\right]  :k \in \left[1,\frac{n-1}{2}\right] \right\}, \\
Y_j &= \left\{ (j-1)\frac{v}{t} + 1 + k \left[(2j-1)\frac{v}{t}+1\right]: k \in \left[0,\frac{n-1}{2}\right]\right\}.
\end{aligned}
\]
It is easy to see that these are sets containing non-zero elements, and that they are disjoint as their elements belong to different classes modulo $\frac{v}{t}$.

Similarly, for every even $j \in [1,t]$ Lemma \ref{lemma:main}(1) can be applied to the $j$-th column of $A$ with $a = -(j-1)\frac{v}{t}-1$, $b = -j\frac{v}{t}+2$, $g = -2$ and $\ell = \frac{n-1}{2}$, obtaining $S(C_j) = X_j \cup Y_j$, where:
\[
\begin{aligned}
X_j &= \left\{ k\left[(1-2j)\frac{v}{t}+1\right]  :k \in \left[1,\frac{n-1}{2}\right] \right\}, \\
Y_j &= \left\{ -(j-1)\frac{v}{t} - 1 + k \left[(1-2j)\frac{v}{t}-1\right]: k \in \left[0,\frac{n-1}{2}\right]\right\}.
\end{aligned}
\]
With the very same argument it can be shown that $X_j$ and $Y_j$ are disjoint sets not containing zero.

One can directly check that, for $\alpha \in \left[1, \frac{n}{t}-1\right]$ and for every $j \in [1,t]$
 the partial sums of  the $j$-th column of $H_{\alpha n}$ are given by $X_{j,\alpha}  \cup Y_{j,\alpha} $ where:
$$X_{j,\alpha}=X_j \quad \textrm{and}\quad Y_{j,\alpha}=Y_j+(-1)^{j+1}\alpha n .$$
It can then be seen  that the sets $X_{j,\alpha}$ and $Y_{j,\alpha}$ are disjoint and contain only non-zero elements.

Hence each column of $A$ admits a simple natural ordering.

In order to prove that the same property holds also for the rows of $A$, we start by observing that the rows of the array $H$ are the following ones.
For every odd $i\in [1,n]$ 
\begin{eqnarray}
\nonumber R_i &=& \Big(i, -\frac{v}{t}-i,2\frac{v}{t}+i, -3\frac{v}{t}-i,\ldots,(t-3)\frac{v}{t}+i, -(t-2)\frac{v}{t}-i, (t-1)\frac{v}{t}+i\Big),
\end{eqnarray}
and for every even $i\in[1,n]$
\begin{eqnarray}
\nonumber R_i &=& \Big(\frac{v}{t}-i, -2\frac{v}{t}+i,3\frac{v}{t}-i, -4\frac{v}{t}+i,\ldots,(t-2)\frac{v}{t}-i, -(t-1)\frac{v}{t}+i,-i\Big).
\end{eqnarray}
For every odd $i\in [1,n]$ we can apply Lemma \ref{lemma:main}(1) with $a=i$, $b=-\frac{v}{t}-i$, $g=2\frac{v}{t}$ and $\ell=\frac{t-1}{2}$, obtaining 
$S(R_i)=X \cup Y_i$ where:
\[
\begin{aligned}
X &= \left\{ -\frac{v}{t}k :k \in \left[1,\frac{t-1}{2}\right] \right\}, \\
Y_i &= \left\{i+k\frac{v}{t}: k \in \left[0,\frac{t-1}{2}\right]\right\}.
\end{aligned}
\]
Analogously, for every even $i\in [1,n]$ we can apply Lemma \ref{lemma:main}(1) with $a=\frac{v}{t}-i$, $b=-2\frac{v}{t}+i$, $g=2\frac{v}{t}$ and $\ell=\frac{t-1}{2}$, obtaining
$S(R_i)=X \cup Y_i$ where:
\[
\begin{aligned}
X &= \left\{ -\frac{v}{t}k :k \in \left[1,\frac{t-1}{2}\right] \right\}, \\
Y_i &= \left\{\frac{v}{t}-i+k\frac{v}{t}: k \in \left[0,\frac{t-1}{2}\right]\right\}.
\end{aligned}
\]
It is immediate to see that, in both cases, $X$ and $Y_i$ are disjoint sets non containing zero.

One can directly check that, for $\alpha \in \left[1, \frac{n}{t}-1\right]$ and
for every $i \in [1,n]$ the partial sums of  the $i$-th row of $H_{\alpha n}$ are given by $X_\alpha \cup Y_{i,\alpha}$ where:

$$X_\alpha=X \quad \text{and} \quad Y_{i,\alpha}=Y_i+(-1)^{i+1}\alpha n.$$
The sets are disjoint and do not contain zero.

Now,  let  $\alpha_1,\alpha_2 \in [0,\frac{n}{t}-1]$ be distinct integers, for any $i \in [1,n]$ and  for $j \in \{1,2\}$ denote by 
$R_{i,\alpha_j} = (a_{i, 1+\alpha_j t}, \, a_{i, 2+\alpha_j t}, \, \dotsc, a_{i, t+\alpha_j t})$ the elements in the $i$-th row of $A$ corresponding to the block $(-1)^\alpha H_{\alpha_j n}$, where $H_0 = H$ is understood. 
By $S(R_{i,\alpha_j})$ we mean the list of the partial sums of $R_i$ corresponding to the subsequence $R_{i,\alpha_j}$,
that is $\{s_a \mid a\in [1+\alpha_j t,t+\alpha_j t] \}$.
To prove that the  rows of $A$ admit a simple natural ordering, we have to show that  $S(R_{i,\alpha_1}) \cap S(R_{i,\alpha_2})= \emptyset$ 
 for every $i\in[1,n]$ and for every choice of $\alpha_1, \, \alpha_2\in [0,\frac{n}{t}-1]$ with $\alpha_1\neq \alpha_2$.

As a first step, we notice that it suffices to control the aforementioned property for $|\alpha_1 - \alpha_2| = 1$. 
In fact, the $i$-th row of $A$ written modulo $\frac{v}{t}$ is:
\[
\begin{aligned} 
R_i \equiv (-1)^{i+1} \big( &\underbrace{i, -i, i, \dotsc, i}_{\text{$t$ terms}}, \underbrace{-(n+i), n+i,  \dotsc, -(n+i)}_{\text{$t$\ terms}},\\& \underbrace{2n+i,-(2n+i), \dotsc, 2n+i}_{\text{$t$\ terms}},\underbrace{-(3n+i), 3n+i,  \dotsc, -(3n+i)}_{\text{$t$\ terms}}, \dotsc,\\ 
& \underbrace{\left(\frac{n}{t}-1\right)n+i,-\left(\left(\frac{n}{t}-1\right)n+i\right), \dotsc, \left(\frac{n}{t}-1\right)n+i}_{\text{$t$\ terms}} \big) ,
\end{aligned}
\]
hence its ordered list of partial sums, considered modulo $\frac{v}{t}$, is given by:
\[
\begin{aligned} 
 (-1)^{i+1} \big( &\underbrace{i, 0, i, \dotsc, i}_{\text{$t$ terms}}, \underbrace{-n, i,  \dotsc, -n}_{\text{$t$\ terms}},
 \underbrace{n+i,-n, \dotsc, n+i}_{\text{$t$ terms}}, \underbrace{-2n, n+i,  \dotsc, -2n}_{\text{$t$ terms}}, \dotsc, \\ 
&\underbrace{\frac{1}{2}\left(\frac{n}{t}-1\right)n+i,-\frac{1}{2}\left(\frac{n}{t}-1\right)n, \dotsc, \frac{1}{2}\left(\frac{n}{t}-1\right)n+i}_{\text{$t$ terms}} \big) .
\end{aligned}
\]
In particular, it can be seen that for every $\alpha \in [0,\frac{n}{t}-1]$ the partial sums of $R_{i,\alpha}$ modulo $\frac{v}{t}$ are:
\[
	\begin{aligned}
	&  (-1)^{i+1}\left( i + \frac{\alpha-1}{2}\right) &  \text{ and } &  (-1)^{i} \frac{\alpha+1}{2}n & \text{for $\alpha$ odd},\\
	&  (-1)^{i+1}\left(i + \frac{\alpha}{2}\right) & \text{ and } &  (-1)^{i}\frac{\alpha}{2}n & \text{for $\alpha$ even}.\\
	\end{aligned}
\]
We can then conclude that for $|\alpha_1 - \alpha_2| > 1$ the partial sums of the $i$-th row of $A$ corresponding to $R_{i,\alpha_1}$  and $R_{i,\alpha_2}$ are trivially disjoint. In the following  for any $i \in [1,n]$ and $\alpha \in \left[0,\frac{n}{t}-1\right]$ set $c=\sum_{j=1}^{\alpha t}a_{i,j}$.
Fix now an odd $i \in [1,n]$, let $\alpha \in \left[0,\frac{n}{t}-1\right]$ be an even integer,
 we have that $S(R_{i,\alpha})  = X_\alpha \cup Y_{i,\alpha}$, where:
\[
\begin{aligned}
X_\alpha &= \left\{ c- k \frac{v}{t} : k \in \left[1, \frac{t-1}{2}\right]  \right\}, \\
Y_{i,\alpha} &= \left\{ c + i +\alpha n +k \frac{v}{t}: k \in \left[0, \frac{t-1}{2} \right] \right\}. \\
\end{aligned}
\]
 It is then not hard to see that $S(R_{i,\alpha+1})  = (X_{\alpha+1} \cup Y_{i,\alpha+1})+(c+i+\alpha n + \frac{t-1}{2}\frac{v}{t})$, where:
\[
\begin{aligned}
X_{\alpha+1} &= \left\{  k \frac{v}{t} : k \in \left[1, \frac{t-1}{2}\right]  \right\}, \\
Y_{i,\alpha+1} &= \left\{ - i -(\alpha+1) n -k \frac{v}{t}: k \in \left[0, \frac{t-1}{2} \right] \right\}. \\
\end{aligned}
\]
We have only to check that $Y_{i,\alpha} \cap \left(X_{\alpha+1} + (c+i+\alpha n + \frac{t-1}{2}\frac{v}{t})\right) = \emptyset $. Assume by contradiction that there exist $k_1 \in  \left[0, \frac{t-1}{2} \right]$, $k_2 \in \left[1, \frac{t-1}{2} \right]$ and an odd $i$ such that:
\[
c + i +\alpha n +k_1 \frac{v}{t} \equiv k_2 \frac{v}{t} + c+i+\alpha n + \frac{t-1}{2}\cdot \frac{v}{t} \pmod{v}.
\]
This can be rewritten as:
\[
\left( \frac{t-1}{2}+k_2-k_1 \right) \frac{v}{t} \equiv 0 \pmod{v}.
\]
It can then be easily seen that this equation cannot hold, as $k_2 - k_1 \in \left[\frac{3-t}{2},\frac{t-1}{2} \right]$.

Now, let $\alpha \in [0,\frac{n}{t}-1]$ be an odd integer; similarly, we have that $S(R_{i,\alpha})  = X_\alpha \cup Y_{i,\alpha}$, where:
\[
\begin{aligned}
X_\alpha &= \left\{ c +k \frac{v}{t} : k \in \left[1, \frac{t-1}{2}\right]  \right\}, \\
Y_{i,\alpha} &= \left\{ c - i -\alpha n -k \frac{v}{t}: k \in \left[0, \frac{t-1}{2} \right] \right\}. \\
\end{aligned}
\]
 It can now be seen that $S(R_{i,\alpha+1})  = (X_{\alpha+1} \cup Y_{i,\alpha+1})+(c-i-\alpha n - \frac{t-1}{2}\frac{v}{t})$, where:
\[
\begin{aligned}
X_{\alpha+1} &= \left\{  -k \frac{v}{t} : k \in \left[1, \frac{t-1}{2}\right]  \right\}, \\
Y_{i,\alpha+1} &= \left\{  i +(\alpha+1) n +k \frac{v}{t}: k \in \left[0, \frac{t-1}{2} \right] \right\}. \\
\end{aligned}
\]
As before, it is sufficient to check that $Y_{i,\alpha} \cap \left(X_{\alpha+1}+(c-i-\alpha n - \frac{t-1}{2}\frac{v}{t})\right) = \emptyset$. Assume by contradiction  that there exist $k_1 \in  \left[0, \frac{t-1}{2} \right]$, $k_2 \in \left[1, \frac{t-1}{2} \right]$ and an odd $i$ such that:
\[
c - i -\alpha n -k_1 \frac{v}{t} \equiv -k_2 \frac{v}{t} + c-i-\alpha n - \frac{t-1}{2}\cdot \frac{v}{t} \pmod{v}.
\]
This can be rewritten as:
\[
\left(\frac{t-1}{2} + k_2 - k_1 \right) \frac{v}{t}  \equiv 0 \pmod{v}.
\]
With the very same argument used before, it can be seen that this equation is never verified, obtaining a contradiction.

Assume now that $i \in [1,n]$ is even, and consider any even  integer $\alpha \in \left[0, \frac{n}{t}-1\right]$. 
We have that $S(R_{i,\alpha})  = X_\alpha \cup Y_{i,\alpha}$, where:
\[
\begin{aligned}
X_\alpha &= \left\{ c- k \frac{v}{t} : k \in \left[1, \frac{t-1}{2}\right]  \right\}, \\
Y_{i,\alpha} &= \left\{ c +\frac{v}{t}- i -\alpha n +k \frac{v}{t}: k \in \left[0, \frac{t-1}{2} \right] \right\}. \\
\end{aligned}
\]
We have then that  
$S(R_{\alpha+1})  = (X_{\alpha+1} \cup Y_{i,\alpha+1})+(c+\frac{v}{t}- i -\alpha n + \frac{t-1}{2}\frac{v}{t})$, where:
\[
\begin{aligned}
X_{\alpha+1} &= \left\{  k \frac{v}{t} : k \in \left[1, \frac{t-1}{2}\right]  \right\}, \\
Y_{i,\alpha+1} &= \left\{ -\frac{v}{t}+ i +(\alpha+1) n -k \frac{v}{t}: k \in \left[0, \frac{t-1}{2} \right] \right\}.\\
\end{aligned}
\]
We need to check $Y_{i,\alpha} \cap \left(X_{\alpha+1}+(c+\frac{v}{t}- i -\alpha n + \frac{t-1}{2}\frac{v}{t})\right) = \emptyset$; assume by contradiction  that there exist $k_1 \in  \left[0, \frac{t-1}{2} \right]$, $k_2 \in \left[1, \frac{t-1}{2} \right]$ and an even $i$ such that:
\[
c +\frac{v}{t}- i -\alpha n +k_1 \frac{v}{t} \equiv  k_2 \frac{v}{t} +c+ \frac{v}{t}- i -\alpha n + \frac{t-1}{2}\cdot\frac{v}{t} \pmod{v}.
\]
This is equivalent to:
\[
\left(\frac{t-1}{2} + k_2 - k_1 \right) \frac{v}{t}  \equiv 0 \pmod{v}.
\]
It is easy to see that this is never the case.

With the very same argument, it can be shown that for every even $i$, the partial sums of $R_{i,\alpha}$ and $R_{i,\alpha+1}$ are distinct 
for every odd $\alpha$.
\end{proof}

\begin{rem}\label{rem:sum_t}
From the globally simple array $A$ constructed in Theorem \ref{thm:t}, one can deduce the total sum of each row and of each column. Let $n,t \geq 1$ be odd integers with $t$ dividing $n$, then for every $j \in [1,t]$ and every $\alpha  \in \left[0, \frac{n}{t}-1\right]$ we have:
\[
\sum_{a_{i,j} \in C_{j+\alpha t}} a_{i,j} = (-1)^{\alpha + j+1}\left[  (j-1)\frac{v}{t} + 1+ \alpha n + \frac{n-1}{2} \left[(2j-1)\frac{v}{t}+1\right] \right].
\]
We now focus on the rows of $A$. Consider any odd $i \in [1,n]$, it can be seen that for every $\alpha \in \left[0, \frac{n}{t}-1\right]$ we have:
\[
\sum_{j \in [1+\alpha t, t+\alpha t]} a_{i,j} = (-1)^{\alpha}\left( i + \alpha n + \frac{t-1}{2}\cdot \frac{v}{t}\right);
\] 
hence the sum of the elements  in the $i$-th row of $A$, for $i$ odd, is:
\[
\sum_{\alpha \in \left[0, \frac{n}{t}-1\right]} (-1)^{\alpha}\left( i + \alpha n + \frac{t-1}{2}\cdot \frac{v}{t}\right).
\]
Thus, it is not hard to see that for every odd $i \in [1,n]$:
\[
\sum_{a_{i,j} \in R_i} a_{i,j} = i+n \left(\frac{n}{t}-1\right) \cdot \frac{1}{2}  + \frac{t-1}{2}\cdot \frac{v}{t}.
\]
Similarly, it can be seen that for every even $i \in [1,n]$ and for every $\alpha \in \left[0, \frac{n}{t}-1\right]$:
\[
\sum_{j \in [1+\alpha t, t+\alpha t]} a_{i,j} = (-1)^{\alpha}\left( \frac{v}{t} - i - \alpha n + \frac{t-1}{2}\cdot \frac{v}{t}\right).
\] 
From which we obtain that for every even $i\in [1,n]$:
\[
\sum_{a_{i,j} \in R_i} a_{i,j} = \frac{v}{t}-i -n \left(\frac{n}{t}-1\right) \cdot \frac{1}{2}  + \frac{t-1}{2}\cdot \frac{v}{t}.
\]
\end{rem}

\begin{ex}
Below we have an $\N\H_{11}(11;11)$, say $A$, whose elements belong to $\Z_{253}$. Following the notation of
the proof of Theorem \ref{thm:t}, we are in the case $t=n$ hence $A=H$.
\begin{center}
\begin{footnotesize}
$\begin{array}{|r|r|r|r|r|r|r|r|r|r|r|}\hline
1 & -24 & 47 & -70 & 93 & -116 & -114 & 91 & -68 & 45 & -22\\ \hline
21 & -44 & 67 & -90 & 113 & 117 & -94 & 71 & -48 & 25 & -2\\ \hline
3 & -26 & 49 & -72 & 95 & -118 & -112 & 89 & -66 & 43 & -20\\ \hline
19 & -42 & 65 & -88 & 111 & 119 & -96 & 73 & -50 & 27 & -4\\ \hline
5 & -28 & 51 & -74 & 97 & -120 & -110 & 87 & -64 & 41 & -18\\ \hline
17 & -40 & 63 & -86 & 109 & 121 & -98 & 75 & -52 & 29 & -6\\ \hline
7 & -30 & 53 & -76 & 99 & -122 & -108 & 85 & -62 & 39 & -16\\ \hline
15 & -38 & 61 & -84 & 107 & 123 & -100 & 77 & -54 & 31 & -8\\ \hline
9 & -32 & 55 & -78 & 101 & -124 & -106 & 83 & -60 & 37 & -14\\ \hline
13 & -36 & 59 & -82 & 105 & 125 & -102 & 79 & -56 & 33 & -10\\ \hline
11 & -34 & 57 & -80 & 103 & -126 & -104 & 81 & -58 & 35 & -12\\ \hline
\end{array}$
\end{footnotesize}
\end{center}

\end{ex}

\begin{ex}
Here we take $n=15$, $t=3,5$ and we follow the proof of Theorem \ref{thm:t}.
An $\N\H_3(15;15)$, say $A$, has elements in $\Z_{453}$, $\frac{v}{t}=151$ and
we have 
$$A=\begin{array}{|r|r|r|r|r|}\hline
H & -H_{15} & H_{30} & -H_{45} & H_{60} \\ \hline
\end{array}$$

\begin{center}
\begin{tiny}
$\begin{array}{|r|r|r|r|r|r|r|r|r|r|r|r|r|r|r|}\hline
1 & -152 & -150 & -16 & 167 & 135 & 31 & -182 & -120 & -46 & 197 & 105 & 61 & -212 & -90\\ \hline
149 & 153 & -2 & -134 & -168 & 17 & 119 & 183 & -32 & -104 & -198 & 47 & 89 & 213 & -62\\ \hline
3 & -154 & -148 & -18 & 169 & 133 & 33 & -184 & -118 & -48 & 199 & 103 & 63 & -214 & -88\\ \hline
147 & 155 & -4 & -132 & -170 & 19 & 117 & 185 & -34 & -102 & -200 & 49 & 87 & 215 & -64\\ \hline
5 & -156 & -146 & -20 & 171 & 131 & 35 & -186 & -116 & -50 & 201 & 101 & 65 & -216 & -86\\ \hline
145 & 157 & -6 & -130 & -172 & 21 & 115 & 187 & -36 & -100 & -202 & 51 & 85 & 217 & -66\\ \hline
7 & -158 & -144 & -22 & 173 & 129 & 37 & -188 & -114 & -52 & 203 & 99 & 67 & -218 & -84\\ \hline
143 & 159 & -8 & -128 & -174 & 23 & 113 & 189 & -38 & -98 & -204 & 53 & 83 & 219 & -68\\ \hline
9 & -160 & -142 & -24 & 175 & 127 & 39 & -190 & -112 & -54 & 205 & 97 & 69 & -220 & -82\\ \hline
141 & 161 & -10 & -126 & -176 & 25 & 111 & 191 & -40 & -96 & -206 & 55 & 81 & 221 & -70\\ \hline
11 & -162 & -140 & -26 & 177 & 125 & 41 & -192 & -110 & -56 & 207 & 95 & 71 & -222 & -80\\ \hline
139 & 163 & -12 & -124 & -178 & 27 & 109 & 193 & -42 & -94 & -208 & 57 & 79 & 223 & -72\\ \hline
13 & -164 & -138 & -28 & 179 & 123 & 43 & -194 & -108 & -58 & 209 & 93 & 73 & -224 & -78\\ \hline
137 & 165 & -14 & -122 & -180 & 29 & 107 & 195 & -44 & -92 & -210 & 59 & 77 & 225 & -74\\ \hline
15 & -166 & -136 & -30 & 181 & 121 & 45 & -196 & -106 & -60 & 211 & 91 & 75 & -226 & -76\\ \hline
\end{array}$
\end{tiny}
\end{center}

An $\N\H_5(15;15)$, say $A$, has elements in $\Z_{455}$, $\frac{v}{t}=91$ and
we have $$A=\begin{array}{|r|r|r|}\hline
H & -H_{15} & H_{30} \\ \hline
\end{array}$$

\begin{center}
\begin{tiny}
$\begin{array}{|r|r|r|r|r|r|r|r|r|r|r|r|r|r|r|}\hline
1 & -92 & 183 & 181 & -90 & -16 & 107 & -198 & -166 & 75 & 31 & -122 & 213 & 151 & -60\\ \hline
89 & -180 & -184 & 93 & -2 & -74 & 165 & 199 & -108 & 17 & 59 & -150 & -214 & 123 & -32\\ \hline
3 & -94 & 185 & 179 & -88 & -18 & 109 & -200 & -164 & 73 & 33 & -124 & 215 & 149 & -58\\ \hline
87 & -178 & -186 & 95 & -4 & -72 & 163 & 201 & -110 & 19 & 57 & -148 & -216 & 125 & -34\\ \hline
5 & -96 & 187 & 177 & -86 & -20 & 111 & -202 & -162 & 71 & 35 & -126 & 217 & 147 & -56\\ \hline
85 & -176 & -188 & 97 & -6 & -70 & 161 & 203 & -112 & 21 & 55 & -146 & -218 & 127 & -36\\ \hline
7 & -98 & 189 & 175 & -84 & -22 & 113 & -204 & -160 & 69 & 37 & -128 & 219 & 145 & -54\\ \hline
83 & -174 & -190 & 99 & -8 & -68 & 159 & 205 & -114 & 23 & 53 & -144 & -220 & 129 & -38\\ \hline
9 & -100 & 191 & 173 & -82 & -24 & 115 & -206 & -158 & 67 & 39 & -130 & 221 & 143 & -52\\ \hline
81 & -172 & -192 & 101 & -10 & -66 & 157 & 207 & -116 & 25 & 51 & -142 & -222 & 131 & -40\\ \hline
11 & -102 & 193 & 171 & -80 & -26 & 117 & -208 & -156 & 65 & 41 & -132 & 223 & 141 & -50\\ \hline
79 & -170 & -194 & 103 & -12 & -64 & 155 & 209 & -118 & 27 & 49 & -140 & -224 & 133 & -42\\ \hline
13 & -104 & 195 & 169 & -78 & -28 & 119 & -210 & -154 & 63 & 43 & -134 & 225 & 139 & -48\\ \hline
77 & -168 & -196 & 105 & -14 & -62 & 153 & 211 & -120 & 29 & 47 & -138 & -226 & 135 & -44\\ \hline
15 & -106 & 197 & 167 & -76 & -30 & 121 & -212 & -152 & 61 & 45 & -136 & 227 & 137 & -46\\ \hline
\end{array}$
\end{tiny}
\end{center}
\end{ex}

\subsection{Globally simple $\N\H_t(n;n)$, for every $n$ prime and admissible $t$}
We start with an easy remark for $n=2$.

\begin{rem}\label{rem:n=2}
Note that the existence of a globally simple $\N\H_t(2;2)$ for any admissible $t$ is trivial.
Below we give an example for each possible case, that is for $t=1,2,4,8$
(the first array is both an $\N\H(2;2)$ and an $\N\H_2(2;2)$):
$$\begin{array}{|r|r|}\hline
 1 & 2 \\\hline
3 & 4\\ \hline
\end{array}
\quad\quad
\begin{array}{|r|r|}\hline
 1 & 2 \\\hline
4 & 5 \\\hline
\end{array}
\quad\quad
\begin{array}{|r|r|}\hline
 1 & 3 \\\hline
5 & 7 \\\hline
\end{array}$$
\end{rem}

Now we can present a complete solution whenever $n$ is a prime.

\begin{thm}\label{thm:prime}
Let $n$ be a prime. There exists a globally simple $\N\H_t(n;n)$ for every admissible $t$.
\end{thm}
\begin{proof}
Since $n$ is a prime the admissible values for $t$ are $1,2,n,2n,n^2,2n^2$.
For $t=1$ the result is contained in Theorem 5.5 of \cite{CDFP}.
For $t \neq 1$ the result follows by Remark \ref{rem:n=2} and by Propositions \ref{prop:2},
\ref{prop:2n}, \ref{prop:n2}, \ref{prop:2n2} and Theorem \ref{thm:t}.
\end{proof}

\section{Connection with path-decompositions}

Simple non-zero sum Heffter arrays are a useful tool for constructing orthogonal path decompositions.
To present this connection we introduce some notations and we recall some basic definitions.
Given a graph $\Gamma$, by $V(\G)$ and $E(\G)$ we denote its vertex-set and its edge-set, respectively.
By $K_v, K_{q \times r}$ and $P_k$ we denote the complete graph on $v$ vertices, the complete multipartite
graph with $q$ parts each of size $r$
and the path of length $k$ (with $k+1$ distinct vertices), respectively. We recall that given
a subgraph $\G$ of a graph $K$, a $\G$-decomposition of $K$ is a set of graphs, called blocks, all isomorphic to $\G$
whose edges partition the edge-set of $K$. Also, a $\G$-decomposition $\D$ of a graph $K$ is said to be cyclic if,
up to isomorphisms, $V(K)=\Z_v$ and for any $B \in \D$, $B+1 \in \D$ too.
Two graph decompositions $\D$ and $\D'$ are said orthogonal if for any $B \in \D$ and any $B' \in \D'$, $|E(B) \cap E(B')|\leq 1$.
The connection between non-zero sum Heffter arrays and graph decompositions is explained by the following result.

\begin{prop}[Proposition 2.9, \cite{CDFP}]\label{prop:dec}
Let $A$ be an $\N\H_t(m,n;h,k)$ simple with respect to the orderings $\omega_{R_i}$ for $i\in[1,m]$ and $\omega_{C_j}$ for $j\in[1,n]$. Then:
  \begin{itemize}
    \item[(1)] there exists a cyclic $P_h$-decomposition $\D_{R}$ of $K_{\frac{2mh+t}{t}\times t}$;
    \item[(2)] there exists a cyclic $P_k$-decomposition $\D_{C}$ of $K_{\frac{2nk+t}{t}\times t}$;
    \item[(3)] the decompositions $\D_{R}$ and $\D_{C}$ are orthogonal.
		\end{itemize}
\end{prop}

The base blocks of the decompositions of previous proposition are so constructed.
Let $\omega_{R_i}$ be a simple ordering of the $i$-th row $R_i$ of $A$, then consider the walk
$P_h^i=(0,s_{i,1},s_{i,2},\ldots, s_{i,h})$, where by $s_{i,j}$ we denote the $j$-th partial sum of $\omega_{R_i}$.
Since the ordering is simple, the partial sums are pairwise distinct and non-zero, hence $P_h^i$  is a path on $h+1$
distinct vertices. The set $\{P_h^i \mid i \in [1,m]\}$ is a set of base blocks of $\D_{R}$. Reasoning in the same way on the columns, we can construct the base blocks of the $P_k$-decomposition  $\D_{C}$.

The existence results proved in the previous section allow us to get the following results on cyclic path decompositions.

\begin{thm}
For every odd integer $n\geq 1$ and every divisor $t$ of $n$, there exists a pair of
orthogonal cyclic $P_n$-decompositions of $K_{\frac{2n^2+t}{t}\times t}$.
\end{thm}
\begin{proof}
The result follows by Theorem \ref{thm:t} and Proposition \ref{prop:dec}.
\end{proof}

\begin{prop}
For every odd integer $n\geq 1$ and any $t\in\{2,2n,n^2,2n^2\}$ there exists a pair of orthogonal cyclic $P_{n}$-decompositions
of $K_{\frac{2n^2+t}{t}\times t}$.
\end{prop}
\begin{proof}
The result follows by Propositions \ref{prop:2}, \ref{prop:2n}, \ref{prop:n2}, \ref{prop:2n2} and  \ref{prop:dec}.
\end{proof}

\begin{rem}
Since Proposition \ref{prop:2} holds for every integer $n$, the previous proposition holds also 
when $n$ is even and $t=2$.
\end{rem}

\section{Connection with biembeddings}
Relative non-zero sum Heffter arrays can be used to construct biembeddings of circuits decompositions on orientable surfaces, as explained in \cite{CDFP}.
In this section, we get new results on biembeddings obtained thanks to the constructions presented in Section 3.
We start by recalling some definitions, see \cite{Moh}.

\begin{defi}\label{defi:embed}
Let $\Gamma$ be a graph endowed with the topology of a $1$-dimensional simplicial complex;
then, an \textit{embedding} of $\Gamma$ in a surface $\Sigma$ is a continuous injective mapping $\phi: \Gamma \rightarrow \Sigma$.
\end{defi}
A connected component of $\Sigma \setminus \phi(\Gamma)$ is called $\phi$-\emph{face}. An embedding $\phi$ is said to be \textit{cellular} if every $\phi$-face is homeomorphic to an open disc.
\begin{defi}
A \emph{biembedding} of two circuit decompositions $\D$ and $\D'$ of a simple graph $\G$ is a face $2$-colorable embedding of $\G$
in which one color class is comprised of the circuits in $\D$ and the other one of the circuits in $\D'$.
\end{defi}

For every edge $e$ of a given graph $\Gamma$, we will consider its two possible directions $e^+$ and $e^-$, and we denote by $\tau$ the involution which swaps $e^+$ and $e^-$ for every $e$. Then, for every vertex $v$ in $\Gamma$, a \textit{local rotation} $\rho_v$ is a cyclic permutation of the edges directed out of $v$. If we choose a local rotation for each vertex of $\Gamma$, then we obtain a rotation of the directed edges of $\G$. We also recall the following result \cite{A,Moh}.

\begin{thm}\label{thm:graph_rot_eq_embed}
A rotation $\rho$ on $\Gamma$ is equivalent to a cellular embedding of $\Gamma$ in an orientable surface. The face boundaries of the embedding corresponding to $\rho$ are the orbits of $\rho \circ \tau$. 
\end{thm}
Moreover, by knowing the number of faces the genus $g$ of the surface can be obtained from Euler's formula $|V| - |E|+ |F| = 2-2g$,
where $V$, $E$, and $F$ denote the number of vertices, edges and faces determined by the embedding in the surface, respectively.

Let $A=\N\H_t(m,n;h,k)$ be a relative non-zero sum Heffter array.
If for any $i\in [1,m]$ and for any $j \in [1,n]$, the orderings $\omega_{R_i}$
and $\omega_{C_j}$ are simple, we define by $\omega_r=\omega_{R_1}\circ \ldots \circ \omega_{R_m}$
the simple ordering for the rows and by $\omega_c=\omega_{C_1}\circ \ldots \circ \omega_{C_n}$
the simple ordering for the columns. 
 The orderings  $\omega_r$ and  $\omega_c$ are said to be
\emph{compatible} if $\omega_c \circ \omega_r$ is a cycle of length  $|\E(A)|$.  
From two compatible orderings a rotation on $K_{\frac{2nk+t}{t} \times t}$ can be obtained, 
implying by Theorem \ref{thm:graph_rot_eq_embed} the existence of a cellular biembedding of this graph. 
Cellular biembeddings obtained from Heffter arrays are also called \textit{Archdeacon embeddings}, see for instance \cite{C}.

Since the existence results on compatible orderings of \cite{CDP} do not depend on the sum of the elements of a row or of a column of the array,
from Theorem 3.3 of \cite{CDP} we have that an $\N\H_t(n;n)$ admits compatible natural orderings if and only if $n$ is odd.

Let $A$ be an $\N\H_t(m, n; h, k)$, we denote by $\D_{\omega_r^{-1}}$ the cyclic  $P_h$-decomposition of $K_{\frac{2nk+t}{t} \times t}$ obtained
from Proposition \ref{prop:dec} with respect to the orderings $\omega_{R_i}^{-1}$, where the starting element is the last element of $\omega_{R_i}$. We recall that, as remarked in Section 2, if $\omega_{R_i}=(a_{i,1},a_{i,2},\ldots, a_{i,k})$ is simple then 
$\omega_{R_i}^{-1}=(a_{i,k},a_{i,k-1},\ldots, a_{i,1})$ is simple too.
Analogously, by  $\D_{\omega_c}$ we mean the cyclic  $P_k$-decomposition of $K_{\frac{2nk+t}{t} \times t}$ obtained
from Proposition \ref{prop:dec} with respect to the orderings $\omega_{C_j}$.

In \cite{CDFP}, the authors construct a decomposition of $K_v$ into circuits starting from a cyclic path decomposition $\D$ of $K_v$,
with $V(K_v)=\Z_v$, as follows. Given $P=(x_0,x_1,\ldots,x_k) \in \D$, let $\lambda_P$ be the minimum positive integer
such that $\lambda_P(x_k-x_0)=0$ modulo $v$. Then define the circuit:
$$C_P:=\cup_{i=0}^{\lambda_P-1} P+i(x_k-x_0).$$
A decomposition of $K_v$ into circuits is given by $\C(\D):=\{C_P : P\in \D\}$.

Then, we report Theorem 6.4 of \cite{CDFP}, which links cellular biembeddings with non-zero sum Heffter arrays:
\begin{prop}\label{prop:nzsh_biemb}
Let A be a non-zero sum Heffter array $\N\H(m, n; h, k)$ that is simple with respect to the compatible orderings $\omega_r$ and $\omega_c$. Then there exists a cellular biembedding of the circuit decompositions $\C(\D_{\omega_r^{-1}})$ and $\C(\D_{\omega_c})$ of $K_{2nk+1}$ into an orientable surface.
\end{prop}

\begin{rem}
We emphasize that if in previous proposition we remove the hypothesis that the orderings are simple we get again a biembedding
(as shown in Theorem 2.5 of \cite{C}), but now the faces may not necessarily be union of paths of length $h$ or $k$.
\end{rem}

We will now see that Proposition \ref{prop:nzsh_biemb} can  be generalized to the case of a \emph{relative} non-zero sum Heffter array.

\begin{prop}\label{prop:GSNZH_biembed}
Let $A$ be a relative non-zero sum Heffter array $\N\H_t(m,n;h,k)$ that is simple with respect  to  two compatible  orderings $\omega_r$ and $\omega_c$. 
 Then there exists a cellular biembedding of two circuit decompositions $\C(\D_{\omega_r^{-1}})$ and $\C(\D_{\omega_c})$ of $K_{\frac{2nk+t}{t} \times t}$ into an orientable surface.
\end{prop}
\begin{proof}
Let $A$ be an $\N\H_t(m,n;h,k)$ admitting two compatible  orderings $\omega_r$ and $\omega_c$, thus $\omega_c \circ \omega_r$ is a cycle of length $|\E(A)|$. Let $\bar{\rho}_0$ be a permutation on
$\pm \E(A) = \Z_{2nk+t}\setminus \frac{2nk+t}{t}\Z_{2nk+t}$ defined by:
\[
\bar{\rho}_0 = \left\{ \begin{aligned}
&-\omega_r(a)\ \text{if $a \in \E(A)$;}\\
&\omega_c(-a)\ \text{if $a \in -\E(A)$;}\\
\end{aligned} \right.
\]
which, as done in Theorem 6.4 of \cite{CDFP}, can be proven to act cyclically on $\pm \E(A)$.

Identify now the complete multipartite graph $K_{\frac{2nk+t}{t} \times t}$ as the Cayley graph $\Cay[\Z_{2nk+t}: \Z_{2nk+t}\setminus \frac{2nk+t}{t}\Z_{2nk+t}]$, that is $\Cay[\Z_{2nk+t}: \pm \E(A)]$.

Define the map $\rho$ on the oriented edges of $K_{\frac{2nk+t}{t} \times t}$ as:
\[
\rho((x,x+a)) = (x,x+\bar{\rho}_0(a)).
\]
It can then be seen that $\rho$ is a rotation of $K_{\frac{2nk+t}{t} \times t}$, since $\bar{\rho}_0$ acts cyclically on $\pm \E(A)$. By Theorem \ref{thm:graph_rot_eq_embed} there exists a cellular embedding of $K_{\frac{2nk+t}{t} \times t}$ in an orientable surface; in particular, its face boundaries correspond to the orbits of $\rho \circ \tau $, where $\tau((x,x+a)) = (x+a,x)$.

Consider now the oriented edge $(x,x+a)$ with $a \in \E(A)$ and $x \in \Z_{2nk+t}$, and let ${C}$ be the column of $A$ containing $a$. Let $\lambda_c$ be the minimum positive integer such that $\sum_{i=0}^{\lambda_c|\E({C})|-1} \omega_c^i(a)=0$ in $\Z_{2nk+t}$. Similarly to Theorem 6.4 of \cite{CDFP}  $(x,x+a)$ belongs to the face $F_1$ whose boundary is:
\[
\left(x,x+a,x+a+\omega_c(a),\ldots,x+\sum_{i=0}^{\lambda_c|\E({C})|-2} \omega_c^i(a)\right).
\]
Moreover, notice that $\lambda_c$ is the minimum positive integer such that $\lambda_c \sum_{a_{i,j}\in C} a_{i,j} = 0$ in $\mathbb{Z}_{2nk+t}$ and $F_1$ has lenght $k\lambda_c$. We can also conclude that $F_1$  covers exactly $\lambda_c$ edges of type $(y,y+a)$ for some $y\in \Z_{2nk+t}$. Since the total number of such edges is $2nk+t$, the column $C$ induces $(2nk+t)/\lambda_c$ different faces.

Let now $a\not \in \E(A)$ and ${R}$ be the row containing $-a$. Consider now the oriented edge $(x,x+a)$ and denote by $\lambda_r$ the minimum positive integer such that $\sum_{i=0}^{\lambda_r|\E({R})|-1} \omega_r^i(-a)=0$ in $\Z_{2nk+t}$. Then, $(x,x+a)$ belongs to the face $F_2$ whose boundary is:
\[
\left(x,x+\sum_{i=1}^{\lambda_r|\E({R})|-1}\omega_{r}^{-i}(-a),x+\sum_{i=1}^{\lambda_r|\E({R})|-2}\omega_{r}^{-i}(-a),
\dots,x+\omega_{r}^{-1}(-a)\right).
\]
Similarly to the previous case, we notice that $\lambda_r$ is the minimum positive integer such that 
$\lambda_r \sum_{a_{i,j}\in R} a_{i,j} = 0$ in $\mathbb{Z}_{2nk+t}$ and $F_2$ has lenght $h\lambda_r$. Moreover, we have that the row $R$ induces $(2nk+t)/\lambda_r$ different faces.

In particular we proved that any non oriented edge $\{x,x+a\}$ of $K_{\frac{2nk+t}{t} \times t}$ belongs to the boundaries of exactly two faces: one of type $F_1$ and one of type $F_2$, thus the embedding is $2$-colorable. Finally note that those face boundaries are circuits of the decompositions $\C(\D_{\omega_r^{-1}})$ and $\C(\D_{\omega_c})$.
\end{proof}

All the non-zero sum Heffter arrays constructed in this paper, for $n$ odd, admits compatible orderings in view of Theorem 3.3 of \cite{CDP}. Now, by following Proposition \ref{prop:GSNZH_biembed}, one can compute the length of each face of the biembeddings, using the total sums reported in Remarks \ref{rem:sum_2}, \ref{rem:sum_2n}, \ref{rem:sum_n2}, \ref{rem:sum_2n2} and \ref{rem:sum_t}. In the next propositions we consider some particular cases where the length of the faces have nice expressions.

We begin by considering the Archdeacon embedding obtained for $t=2$:
\begin{prop}\label{prop:2_length_faces}
Let $n\geq 1$ be an odd integer, then there exists a cellular biembedding of $K_{(n^2+1) \times 2}$ such that every face has length
 $4n$ or a multiple of $n\frac{n^2+1}{2}$.
\end{prop}
\begin{proof}
Let $n$ be an odd positive integer and $A$ be the $\N\H_2(n;n)$ constructed in Proposition \ref{prop:2}; it will be enough to prove that the total sum of every row and of every column has order $4$ or a multiple of $\frac{n^2+1}{2}$ modulo $2(n^2+1)$.
By Remark \ref{rem:sum_2}, for every $j \in [1,n]$:
\[
\sum_{a_{i,j} \in C_j} a_{i,j}  =(-1)^{j+1} \left( j+  \frac{n-1}{2}\,n  \right).
\]
Note that for $j = \frac{n+1}{2}$ the previous expression yields $\pm \frac{n^2+1}{2}$, obtaining order $4$ modulo $2(n^2+1)$. For $j \neq \frac{n+1}{2}$ it can be seen that the greatest common divisor between the total sum and  $n^2+1$ is either $1$ or $2$, obtaining that the order is a multiple of $\frac{n^2+1}{2}$.

Now, again by Remark \ref{rem:sum_2}, for every $i \in [1,n]$:
\[
\sum_{a_{i,j} \in R_i} a_{i,j} =(-1)^{i+1} \left( (i-1)n  + \frac{n+1}{2} \right).
\]
Similarly to the previous case, for $i= \frac{n+1}{2}$ we obtain order $4$, while for $i \neq \frac{n+1}{2}$ the order is a multiple of $\frac{n^2+1}{2}$.
\end{proof}

\begin{prop}\label{prop:2n_length_faces}
Let $n$ be an odd prime, then there exists a cellular biembedding of $K_{(n+1)\times 2n}$ such that every face has length $4n$  or a multiple of $n^2$.
\end{prop}
\begin{proof}
Let $A$ be the $\N\H_{2n}(n;n)$ described in Proposition \ref{prop:2n}. By Remark \ref{rem:sum_2n}, we have that for every $j \in [1,n]$:
\[
\sum_{a_{i,j} \in C_j} a_{i,j} =(n+1)\left(\frac{n^2}{2} +(-1)^j\,\frac{n+1-2j}{2}\right).
\]
As $n$ is prime and $j \in [1,n]$, it can be seen that for $j \neq \frac{n+1}{2}$ the rightmost bracket is coprime with $n$; hence, the order in $\Z_{2n(n+1)}$ is $4n$, so we obtain an embedding where the length of the corresponding faces is $4n^2$. For $j = \frac{n+1}{2}$, the previous expression reads $n^2(n+1)/2$, obtaining order $4$, thus the length of the corresponding faces is $4n$.

We will now prove that for every row of $A$ the order of its total sum in $\Z_{2n^2+2n}$ is either $4$ or a multiple of $n$. 
For every odd $i \in [1,n]$:
\[
\sum_{a_{i,j} \in R_i} a_{i,j} = \frac{-n^3+n^2+n-1+2i}{2}.
\]
Firstly, since $n^2+2n = -n^2$ in $\Z_{2n^2+2n}$, for $i = \frac{n+1}{2}$ the previous expression reads $\frac{-n^3-n^2}{2} = \frac{-n^2(n+1)}{2}$, thus obtaining order $4$ and corresponding length of the faces $4n$. For $i \neq \frac{n+1}{2}$  the total sum is not a multiple of $n$, as $\frac{-n^3+n^2}{2}$ is a multiple of $n$, while $\frac{n-1+2i}{2}$ is not. Thus, for $i \neq \frac{n+1}{2}$, the order of the total sum is a multiple of $n$.

Finally, for every even $i \in [1,n]$:
\[
\sum_{a_{i,j} \in R_i} a_{i,j} = \frac{n^3+n^2+n+1-2i}{2}.
\]
As for the previous case, it can be seen that the order of the total sum of the $i$-th row of $A$  is a multiple of $n$ for $i \neq \frac{n+1}{2}$, and it is $4$ for $i = \frac{n+1}{2}$.
\end{proof}

\begin{prop}\label{prop:n2_length_faces}
Let $n$ be an odd prime, then there exists a cellular biembedding of $K_{3\times n^2}$ such that every face has length $3n^2$ or $3n^3$.
\end{prop}
\begin{proof}
Let $A$ be the $\N\H_{n^2}(n;n)$ described in Proposition \ref{prop:n2}; we follow Remark \ref{rem:sum_n2} where the total sums of rows and columns are described. In particular, for every $j \in [1,n]$:
\[
\sum_{a_{i,j} \in C_j} a_{i,j} =(-1)^{j-1} \left( 3nj-\frac{3n+1}{2} \right).
\]
Now, since $\frac{3n+1}{2}$ is neither a multiple of $n$, nor a multiple of $3$, we obtain that $\sum_{a_{i,j} \in C_j} a_{i,j}$ and $3n^2$ are coprime, thus every  face constructed from the columns of $A$ has length $3n^3$.

If $n \equiv 1 \pmod{4}$, for every $i \in [1,n]$:
\[
\sum_{a_{i,j} \in R_i} a_{i,j} = (-1)^{i-1} \left(3i - 2 +3n \frac{n-1}{2}\right).
\]
It can then be seen that this term is  coprime with $3n^2$ except when $3i$ is equal to $n+2$ (that implies $n \equiv 1 \pmod{3}$)
or $2n+2$ (that may hold if $n \equiv 2 \pmod{3}$). Let $\bar{i}$ denote such an index, note that $3n \sum_{a_{\bar{i},j} \in R_{\bar{i}}} a_{\bar{i},j} = 0 \pmod{3n^2}$.

For every  $n \equiv 3 \pmod{4}$ and $i \in [1,n]$:
\[
\sum_{a_{i,j} \in R_i} a_{i,j} = (-1)^{i-1} \left(3(n+1-i) - 2 +3n \frac{n-1}{2}\right).
\]
Similarly, this term is coprime with $3n^2$  except when $3i$ is  equal to $n+1$ (that implies $n \equiv 2 \pmod{3}$)
or $2n+1$ (that may hold if $n \equiv 1 \pmod{3}$). Let $\bar{i}$ be such an index, then
$3n \sum_{a_{\bar{i},j} \in R_{\bar{i}}} a_{\bar{i},j} = 0 \pmod{3n^2}$.

In any case, we have proven that the length of the faces obtained by the $\bar{i}$-th row is $3n^2$, while every other face has length $3n^3$.
\end{proof}

\begin{prop}\label{prop:2n2_length_faces}
Let $n$ be an odd prime, then there exists a cellular biembedding of $K_{2\times 2n^2}$ such that every face has length $4n$, $4n^2$ or $4n^3$.
\end{prop}
\begin{proof}
Let $A$ be the $\N\H_{2n^2}(n;n)$ described in Proposition \ref{prop:2n2}; we recall that the total sum of the rows and the columns of $A$ is given in  Remark \ref{rem:sum_2n2}. For every $j \in [1,n]$ we have:
\[
\sum_{a_{i,j} \in C_j} a_{i,j} =n^3 + (-1)^{j-1} (2nj-n^2-n).
\]
It can be seen that  for  $j \neq \frac{n+1}{2}$, the order of the sum in $\Z_{4n^2}$ is $4n$. For $j =  \frac{n+1}{2}$ the previous expression yields $n^3$, which has order $4$ in $\Z_{4n^2}$. Thus the corresponding faces have length $4n^2$ or $4n$.

For every $n \equiv 1 \pmod{4}$ and $i \in [1,n]$ we have:
\[
\sum_{a_{i,j} \in R_i} a_{i,j} =n^3 + (-1)^{i-1} (2i-n-1).
\]
It can be seen that for $i= \frac{n+1}{2}$ the previous expression yields $n^3$, thus the corresponding faces have length $4n$. For $i \in [1,n]$, $i \neq \frac{n+1}{2}$, the previous expression gives a number coprime with $4n^2$, obtaining that the faces constructed from these rows have length $4n^3$.

For every $n \equiv 3 \pmod{4}$ and  for every $i \in [1,n]$:
\[
\sum_{a_{i,j} \in R_i} a_{i,j} =n^3 + (-1)^{i-1} \left(n+1-2i\right).
\]
It is then an easy check to verify that this term is coprime with $4n^2$ for every $i \in [1,n]$, $i \neq  \frac{n+1}{2}$, obtaining faces with length $4n^3$, while for $i = \frac{n+1}{2}$ the faces have length $4n$.
\end{proof}

\begin{prop}\label{prop:n_length_faces}
Let $n \geq 1$ be an odd integer such that $2n+1$ is prime, then there exists a cellular biembedding of  $K_{(2n+1)\times n}$ such that the length of every face is a multiple of $n(2n+1)$.
\end{prop}
\begin{proof}

Let $A$ be the $\N\H_n(n;n)$ obtained from Theorem \ref{thm:t}, we follow Remark \ref{rem:sum_t} where we reported the total sums of rows and columns of $A$, keeping in mind that $t = n$ which implies $\alpha  =0$. In particular, for every $j \in [1,n]$:
\[
\sum_{a_{i,j} \in C_{j}} a_{i,j} = (-1)^{ j+1}\left[  (j-1)(2n+1)+ 1 + \frac{n-1}{2} \left[(2j-1)(2n+1)+1\right] \right].
\]
Since we are in $\Z_{2n^2+n}$, the previous expression is equivalent to:
\[
\sum_{a_{i,j} \in C_{j}} a_{i,j} = (-1)^{ j+1} (-n^2-n) = (-1)^j n(n+1).
\]
We can then conclude by noticing that $n+1$ and $2n+1$ are coprime, thus obtaining that the total sum of the elements in every column has order $2n+1$ in $\Z_{2n^2+n}$ and that the corresponding faces of the Archdeacon embedding have length $n(2n+1)$.

Now, we have for every odd $i \in [1,n]$:
\[
\sum_{a_{i,j} \in R_i} a_{i,j}= i + \frac{n-1}{2} (2n+1).
\]
Since $2n+1$ is prime, it is immediate to see that this expression and $2n+1$ are coprime, thus the length of the corresponding faces is a multiple of $n(2n+1)$.
Similarly, for every even $i \in [1,n]$:
\[
\sum_{a_{i,j} \in R_i} a_{i,j}= -i  + \frac{n+1}{2} (2n+1).
\]
As before, it can be seen that this expression is coprime with $2n+1$, obtaining as length face multiples of $n(2n+1)$.
\end{proof}

\begin{rem}
Since all the non-zero sum Heffter arrays constructed in this paper have the additional property of being globally simple,
the faces of biembeddings constructed in Propositions \ref{prop:2_length_faces}, \ref{prop:2n_length_faces}, \ref{prop:n2_length_faces}, 
\ref{prop:2n2_length_faces}  and \ref{prop:n_length_faces}
are union of paths of length $n$.
\end{rem}

\section*{Acknowledgements}
The authors would like to thank Simone Costa for helpful discussions on this topic.
The authors were partially supported by INdAM-GNSAGA.


\begin{thebibliography}{50}

\bibitem{A} D.S. Archdeacon,
\textit{Heffter arrays and biembedding graphs on surfaces},
Electron. J. Combin. \textbf{22} (2015) \#P1.74.

\bibitem{ADDY} D.S. Archdeacon, J.H. Dinitz, D.M. Donovan \and E.\c{S}. Yaz\i c\i,
\textit{Square integer Heffter arrays with empty cells},
Des. Codes Cryptogr. \textbf{77} (2015), 409--426.

\bibitem{BCDY} K. Burrage, N.J. Cavenagh, D. Donovan \and  E.\c{S}. Yaz\i c\i,
\textit{Globally simple Heffter arrays $H(n;k)$ when $k\equiv 0,3 \pmod{4}$},
Discrete Math. \textbf{343} (2020), 111787.

\bibitem{CDDY} N.J. Cavenagh, J. Dinitz, D. Donovan \and  E.\c{S}. Yaz\i c\i,
\textit{The existence of square non-integer Heffter arrays},
Ars Math. Contemp. \textbf{17} (2019), 369--395.

\bibitem{CDY} N.J. Cavenagh, D. Donovan \and  E.\c{S}. Yaz\i c\i,
\textit{Biembeddings of cycle systems using integer Heffter arrays},
J. Combin. Des. \textbf{28} (2020), 900--922.

\bibitem{C} S. Costa, \textit{Biembeddings of Archdeacon type: their full automorphism group and their number},
preprint available at https://arxiv.org/abs/2205.02066.


\bibitem{CDP} S. Costa, M. Dalai \and A. Pasotti,
\textit{A tour problem on a toroidal board},
Austral. J. Combin. \textbf{76} (2020), 183--207.

\bibitem{CDFOR} S. Costa, S. Della Fiore, M. A. Ollis \and  S. Z. Rovner-Frydman,
\textit{On Sequences in Cyclic Groups with Distinct Partial Sums},
preprint available at https://arxiv.org/abs/2203.16658.

\bibitem{CDFP} S. Costa, S. Della Fiore \and A. Pasotti, \textit{Non-zero sum Heffter arrays and their applications},
Discrete Math. \textbf{345} (2022), 112952.

\bibitem{CMPPHeffter} S. Costa, F. Morini, A. Pasotti \and M.A. Pellegrini,
\textit{Globally simple Heffter arrays and orthogonal cyclic cycle decompositions},
Austral. J. Combin. \textbf{72} (2018), 549--593.

\bibitem{CPPBiembeddings} S. Costa, A. Pasotti \and M.A. Pellegrini,
\textit{Relative Heffter arrays and biembeddings},
 Ars Math. Contemp. \textbf{18} (2020), 241--271.

\bibitem{CPJCTA} S. Costa \and A. Pasotti,\textit{On the number of non-isomorphic (simple) k-gonal biembeddings of complete multipartite graphs},
preprint available at https://arxiv.org/abs/2111.08323.


 \bibitem{CPEJC} S. Costa \and A. Pasotti, \textit{On $\lambda$-fold relative Heffter arrays and biembedding multigraphs on surfaces},
Europ. J. Combin. \textbf{97} (2021), 103370.

\bibitem{DM} J.H. Dinitz \and A.R.W. Mattern,
\textit{Biembedding Steiner triple systems and $n$-cycle systems on orientable surfaces},
Austral. J. Combin. \textbf{67} (2017), 327--344.

\bibitem{DW} J.H. Dinitz \and I.M. Wanless,
\textit{The existence of square integer Heffter arrays},
Ars Math. Contemp. \textbf{13} (2017), 81--93.

\bibitem{M} L. Mella, \textit{Square globally simple non-zero sum Heffter arrays with empty cells}, in preparation.

\bibitem{Moh} B. Mohar,
\textit{Combinatorial local planarity and the width of graph embeddings},
Canad. J. Math. \textbf{44} (1992), 1272--1288.

\end{thebibliography}
\end{document}